\documentclass[10 pt,a4paper]{article}
\usepackage[toc]{appendix}
\usepackage{amsmath}
\usepackage{tocbibind}
\usepackage{stmaryrd}              
\usepackage{verbatim}
\usepackage{comment}
\usepackage{mathrsfs}
\usepackage{mathtools}
\usepackage{amsthm} 
\usepackage{amssymb} 
\usepackage{mdframed}
\usepackage[margin=2.5 cm]{geometry}
\usepackage{enumerate}
\usepackage{bm}
\usepackage[colorlinks={true},linkcolor={blue},citecolor={blue}, urlcolor={blue}]{hyperref}
\usepackage[noabbrev,nameinlink]{cleveref}
\usepackage{comment}
\usepackage{tikz}
\usepackage{tikz-cd}
\usetikzlibrary{intersections}
\usetikzlibrary{spy} 
\theoremstyle{plain}
\usepackage[textsize=footnotesize]{todonotes}

\usepackage{calligra}

\newtheorem{thm}{Theorem}[section]
\newtheorem*{thm*}{Theorem}
\newtheorem{prop}[thm]{Proposition}
\newtheorem{cor}[thm]{Corollary}
\newtheorem{lem}[thm]{Lemma}
\theoremstyle{definition}
\newtheorem{defn}[thm]{Definition}
\newtheorem{conj}[thm]{Conjecture}
\newtheorem{expl}[thm]{Example}

\newtheorem{rem}[thm]{Remark}

\newtheoremstyle{named}%
    {}{}{\itshape}{}{\bfseries}{.}{.5em}{\thmnote{#3}}
\theoremstyle{named}

\DeclareMathOperator{\Aut}{Aut}
\DeclareMathOperator{\tor}{tor}

\DeclareMathOperator{\Tr}{Tr}
\DeclareMathOperator{\ur}{ur}

\newcommand{\bA}{\mathbb A}
\newcommand{\cA}{\mathcal A}
\newcommand{\fA}{\mathfrak A}
\newcommand{\bC}{\mathbb C}
\newcommand{\cE}{\mathcal E}
\newcommand{\cF}{\mathcal F}

\newcommand{\bH}{\mathbb H}

\newcommand{\cO}{\mathcal O}

\newcommand{\cP}{\mathcal P}
\newcommand{\fP}{\mathfrak P}
\newcommand{\bR}{\mathbb R}
\newcommand{\bQ}{\mathbb Q}
\newcommand{\bZ}{\mathbb Z}
\DeclareMathOperator{\an}{an}

\DeclareMathOperator{\Lie}{Lie}

\newcommand{\cLie}{\text{\calligra Lie}\,}

\title{Relative monodromy of ramified sections on abelian schemes}
\author{Paolo Dolce \and 
Francesco Tropeano}
\date{}

\newcommand{\Addresses}{{
  \bigskip
  \bigskip
  \footnotesize
}

  P.~Dolce, \textsc{Institute for Theoretical Sciences, Westlake University, China}\par\nopagebreak
  \textit{E-mail address}: \texttt{dolce@westlake.edu.cn}

  \bigskip
  
F.~Tropeano, \textsc{Universit\`a degli Studi Roma Tre, Italy}\par\nopagebreak
  \textit{E-mail address}: \texttt{francesco.tropeano@uniroma3.it}
  
}

\begin{document}

\maketitle

\makeatletter
\@starttoc{toc}
\makeatother

\begin{abstract}
    Let's fix a complex abelian scheme $\cA\to S$ of relative dimension $g$, without fixed part, and having maximal variation in moduli. We show that the relative monodromy group $M^{\textrm{rel}}_\sigma$ of a ramified section $\sigma\colon S\to\mathcal A$ is nontrivial. Moreover, under some hypotheses on the action of the monodromy group $\textrm{Mon}(\mathcal A)$ we show that $M^{\textrm{rel}}_\sigma\cong \mathbb Z^{2g}$. We discuss several examples and applications. For instance we provide a new proof of Manin's kernel theorem and of the algebraic independence of the coordinates of abelian logarithms with respect to the coordinates of periods.
\end{abstract}

\section{Introduction}
\paragraph{Notation and Background.}
We establish the following notation, which will remain fixed throughout this paper. Let \( S \) be a regular, irreducible, quasi-projective variety over \(\mathbb{C}\), and let \(\phi \colon \cA \to S\) be an abelian scheme over \(\mathbb{C}\) of relative dimension \(g\). By definition, this means that \(\cA\) is a proper, smooth group $\mathbb C$-scheme over \(S\) equipped with a zero section \(\sigma_0 \colon S \to \cA\), and whose fibers \(\mathcal{A}_s\) are abelian varieties of dimension \(g\) for every \(s \in S\). The generic fiber of $\mathcal A$ is denoted by $A$ and  is an abelian variety over the complex function field $\mathbb C(S)$. Since any rational section of $\phi\colon\mathcal A\to S$ can be extended to a global section (see for instance \cite[Proposition 6.2]{LiuLor}),  there is a group isomorphism between the rational points $A(\mathbb C(S))$ and $\Sigma(S)$ which denotes the set of (algebraic) sections of $\phi$. We recall that the \emph{fixed part of $A$} (sometimes called the ``constant part") is the image   of the $\overline{\mathbb C(S)}/\mathbb C$-trace of $A\otimes\overline{\mathbb C(S)}$ (see for instance \cite{Con}). The fixed part is essentially the maximal abelian subcheme of $A$ defined over the base field $\mathbb C$, hence ``constant'' when spread over  $S$. We say that the abelian scheme $\mathcal A$ is \emph{without fixed part} if the fixed part of $A$ is $\{0\}$.

We denote with  \(\mathbb{A}_{g,\ell}\) the moduli space of $g$-dimensional principally polarized complex abelian varieties endowed with a level-$\ell$-structure, with $\ell\ge 3$. Since  \(\mathbb{A}_{g,\ell}\) is a fine moduli space, it comes with a universal family \(\mathfrak{A}_{g,\ell} \to \mathbb{A}_{g,\ell}\). To simplify notation, we shall omit the subscript \(\ell\) in what follows, writing instead \(\mathfrak{A}_g \to \mathbb{A}_g\). Furthermore, up to replacing $\mathcal A\to S$ with a base change by a finite \'etale morphism $S'\to S$, by \cite[Proof of Theorem B.1 Devissages $(iv)-(vi)$]{DGH}, we can assume that the abelian scheme carries principal polarization and it has level-$\ell$-structure. Therefore, we fix a modular morphism \(p \colon S \to \mathbb{A}_g\) induced by $\phi$, and denote its image by \(T = p(S)\).

For any complex variety $X$ we denote by $X(\mathbb C)$ the associated analytification in the sense of \cite[Exposé XII]{SGA1}. If $X$ is regular, then $X(\mathbb C)$ is a complex manifold with the associated locally euclidean topology. We denote with $\Sigma^{\an}$ the sheaf over $S(\mathbb C)$ of holomorphic sections of $\phi$ whereas $\Sigma(S)\subset \Sigma^{\an}(S(\mathbb C))$ is the abelian group of (global) algebraic $\mathbb C$-sections. Each fiber $\mathcal A_s(\mathbb C)$ can be analytically identified with a complex torus $\mathbb C^g/\Lambda_s$. Locally, on simply connected subsets  $U\subset S(\mathbb C)$ it is possible to define a holomorphic period map $\mathfrak P: U\to\Lie(\cA)^{2g}$ where $\Lie(\cA)$ is the complex Lie algebra bundle of $\cA(\mathbb C)$. The period map associates to any point $s\in U$ a basis $\mathfrak P(s)$ of the lattice $\Lambda_s$ and in general it  cannot be globally defined on $S(\mathbb C)$. The obstruction on the existence of the analytic continuation of $\mathfrak P$   is measured by the \emph{monodromy group (of $\mathcal A$)}, $\textrm{Mon}(\mathcal A)\subseteq\text{Sp}_{2g}(\mathbb Z)$. The group $\textrm{Mon}(\mathcal A)$ is often addressed also as \emph{monodromy of periods}. Given an algebraic section $\sigma\colon S\to \cA$ of the abelian scheme,  locally on simply connected open subsets $U'\subset S(\mathbb C)$ it is possible to lift $\sigma$ through the exponential map $\exp\colon \Lie(\cA)\to \cA(\mathbb C)$ to a holomorphic map $\log_\sigma\colon U'\to \Lie(\cA)$. This map is called an abelian logarithm of the section $\sigma$. On a common open set $V\subset S(\mathbb C)$ where both $\log_\sigma$ and $\mathfrak P$ are defined, we can express $\log_\sigma(s)\in\mathbb C^{g}$ in terms of the basis $\mathfrak P(s)$ of the lattice $\Lambda_s$. The coordinates form a real-analytic function $\beta_\sigma\colon V\to\mathbb R^{2g}$ which is nowadays called the Betti map (associated to $\sigma$) and that turned out to be of crucial importance in the theory of unlikely intersections. Such map was implicitly considered by Manin in \cite{Man}, to get a proof of Mordell's conjecture in the function field case. The  abelian logarithm $\log_\sigma$ is in general defined only locally, thus it is interesting to measure the obstruction to the existence of a global abelian logarithm: the simultaneous monodromy of the period map and  $\log_\sigma$ gives rise to the \emph{monodromy group of the section}, $M_\sigma\subseteq\textrm{SL}_{2g+1}(\mathbb Z)$.\\

The non-triviality of $\textrm{Mon}(\cA)$ and $M_\sigma$ gives relevant information on the structure of the abelian scheme and on the section, with respect to the fixed part. In fact: no non-zero periods can be globally defined (so in particular  $\textrm{Mon}(\cA)\neq \{0\}$) when the abelian scheme has no fixed part (see \cite[Lemma 5.6]{GH2}). In addition,  if  $M_\sigma =\begin{psmallmatrix} \mathrm{Mon}(\mathcal{A}) & 0 \\ 0 & 1 \end{psmallmatrix}$ then $\sigma$ is contained in the fixed part of $\mathcal A$ (see \Cref{wellDefLog}).

A deeper understanding of monodromy group structures, going beyond simply establishing their non-triviality, is often essential. For example, $\textrm{Mon}(\mathcal{A})\subseteq\textrm{Sp}_{2g}(\mathbb Z)$ measures the ``modular complexity'' of $\phi\colon\mathcal{A}\to S$, and a key question is whether it is \emph{thin}, i.e. of infinite index in $G(\mathbb{Z})$, where $G$ denotes its Zariski closure.  Works such as \cite{PONCELET} and \cite{NORI} exhibit important cases where $\textrm{Mon}(\mathcal{A})$ is thin. While thin monodromy groups appear to be more common in certain settings, non-thin examples also arise (see, e.g., \cite{SARNAK} or \cite[Section 2.1]{Ell} for a discussion).  For a finitely generated subgroup of $\textrm{GL}_n(\mathbb{Z})$, determining its Zariski closure is usually computationally feasible. However, verifying whether the group is thin can be a delicate and highly non-trivial problem.

\paragraph{Main results of this paper.} We investigate the obstruction that prevents the existence of a globally defined abelian logarithm along loops which induce a trivial monodromy action on the period map. This obstruction is measured by the \emph{relative monodromy group} (of the section) $M^{\textrm{rel}}_\sigma\subseteq \mathbb Z^{2g}$.
The group $M_\sigma^\textrm{rel}$ gives relevant information on the section, in particular its non-triviality is supposed be related to the property of $\sigma$ of being non-torsion, and its rank should give information on the minimal group subscheme containing the image of $\sigma$. 

In the case of non-isotrivial elliptic schemes over a curve Corvaja and Zannier proved in \cite{CZ} that when $\sigma$  is a non-torsion  section the group $M^{\textrm{rel}}_\sigma$ is non-trivial and moreover it is isomorphic to $\mathbb Z^2$ (so it has full rank). In the same paper, they provided an example showing that a priori, if we just limit ourselves to group theoretic arguments, the relative monodromy group could be smaller than expected (even trivial) when compared to the relative group arising from Zariski closures of non-relative monodromy groups. Hence, in this setting there are some groups which cannot arise as monodromy groups of sections. We point out that the second author of the present paper in \cite{T2} gives an explicit proof of the results of Corvaja and Zannier. 

In this paper we generalize the work of \cite{CZ} on the relative monodromy $M^{\textrm{rel}}_\sigma$, focusing on ramified sections instead of non-torsion sections. The concept of a \emph{ramified section} is formally defined below. Intuitively, a section $\sigma \colon S \to \cA$ is ramified if it does not arise from a base change of an abelian scheme $\cA' \to S'$ whose modular map is unramified and factors $p$. The main technical difficulty arises when $T$ is non-normal, requiring a more careful definition.

Let $\nu \colon T^\nu \to T$ denote the normalization of $T$, and let $\mathfrak{A}_{g,T^\nu} \to T^\nu$ be the base change of the universal family to $T^\nu$. By the universal property of normalization, the modular map $p \colon S \to T$ factors uniquely through $\nu$, yielding a dominant map $p^\nu \colon S \to T^\nu$. Let $L$ be the maximal unramified extension\footnote{Let $X$ be an integral, normal scheme. Recall that an extension $L$ of $K(X)$ is said unramified if the normalization $X'\to X$ of $X$ in $L$ is an unramified morphism.} of $\mathbb C(T^\nu)$ in $\mathbb C(S)$ and let $S^{\ur}$ be the normalization of $T^\nu$ in $L$.  By the universal property of the normalization (in a field extension) we get a factorization
\begin{equation}\label{eq:pnu}
\begin{tikzcd}
S\arrow{r}{q}\arrow[bend left]{rr}{p_\nu} & S^{\ur} \arrow{r}{q'\;\;} & T^\nu
\end{tikzcd}
\end{equation}
where $q'$ is unramified, $q$ is dominant and $S^{\ur}$ is normal. By taking the fiber product over $T^\nu$ we fix once for all the following commutative diagram
\begin{equation}\label{eq:relNormalization}
	\begin{tikzcd}
		\cA \arrow{d} \arrow{r} & \cA' \arrow{d} \arrow{r} & \fA_{g,T^\nu} \arrow{d}\\
		S \arrow{r}{q} \arrow[bend right, swap]{rr}{p^\nu}& S^{\ur} \arrow{r}{q'} & T^\nu.
	\end{tikzcd}
\end{equation}
\begin{defn}\label{def:ram_sect}
    With the above notation fixed we say that an algebraic section $\sigma:S \to \cA$ is \emph{unramified} (over $S$) if there exists an algebraic  section $\sigma':S^{\ur} \to \cA'$ such that $\sigma=q^*\sigma'$, where 
\begin{eqnarray*}
q^*\sigma'\colon S(\mathbb C)& \to& \mathcal A(\mathbb C)\\
s &\mapsto&(\mathfrak q^{-1}_s\circ\sigma'\circ q)(s)
\end{eqnarray*}
and $\mathfrak q_s:\mathcal A_s\xrightarrow{\cong} \mathcal A_{q(s)}$  is the isomorphism induced by the fiber product, for any $s\in S(\mathbb C)$. 
An algebraic section $\sigma\colon S\to \cA$ is said to be \emph{ramified} (over $S$) if it is not unramified. We say that $\sigma$ is unramified [resp. ramified] over a Zariski open subset $U\subseteq S$ if $\sigma_{|U}$ is an unramified [resp. ramified] section  of the restricted abelian scheme $\mathcal A_{|U}\to U$.
\end{defn}

In  \Cref{{cor:divisible_ram}} we show that every torsion section is unramified. This fact is noted in \cite{CZ} for elliptic schemes over a curve, though we were unable to find a proof or reference for higher-dimensional cases. Of particular interest are settings where unramified sections are necessarily torsion, as this condition is equivalent to say that every abelian scheme obtained via an unramified base change of $\fA_{g,T} \to T$ has zero Mordell-Weil rank. For example, Shioda \cite{Shio} proves that unramified sections of elliptic schemes over a curve are always torsion, while Mok and To \cite{MokTo} establish the vanishing of the Mordell-Weil rank for unramified finite extension Kuga families (without locally constant part). However, in general torsion sections form a proper subset of unramified sections, as in \Cref{expl:tor_ram} we construct a non-torsion unramified section.

In the following, we consider abelian schemes without fixed part. Additionally, we assume that the modular map $p:S\to T\subseteq\mathbb A_g$ is generically finite\footnote{For us a morphism of schemes $f\colon X\to Y$ is \emph{generically finite} if there exists a dense open $V\subseteq Y$ such that the restriction $f\colon f^{-1}(V)\to V$ is a finite morphism.} (onto its image), this assumption is often referred as \emph{maximal variation in moduli}. When $\mathcal A \to S$ is simple and $S$ is a curve, maximal variation in moduli is equivalent to being without fixed part. We prove the result stated below.



\begin{thm}\label{mainThm1}
    Let $\phi:\mathcal A \to S$ be an abelian scheme without fixed part such that the modular map $p:S \to T \subseteq \mathbb A_g$ is generically  finite. Let $U_p\subseteq S$ be the Zariski open subset on which $p$ is finite and flat. Then for any section $\sigma\colon S\to \mathcal A$  ramified over $U_p$, the relative monodromy group $M_\sigma^{\textrm{rel}}$ is non-trivial.
\end{thm}
When $\phi\colon\mathcal{A} \to S$ is an elliptic scheme over a curve, we may always assume that the modular map $p$ is finite and flat after performing a base change. In this case we have $U_p = S$, and the results of \cite{CZ} are recovered. An interesting case where our result holds, is given by the family of Jacobians of hyperelliptic curves described in \cite[Section 8.1]{ACZ}. Our techniques seem unsuitable to prove the same statement for all non-torsion sections or for a more general modular map $p$. However, in the particular case  when $p$ is finite and flat and all unramified sections of $\mathcal A\to S$ are torsion, then the claim \Cref{mainThm1} clearly holds for all non-torsion sections. For instance, by \cite[Main Theorem]{MokTo} such special situation occurs  when $p$ is finite and $\mathfrak A_{g,T}\to T$ is a Kuga family without locally constant  part. In particular, $T$ may be chosen to be a totally geodesic (i.e. weakly-special) subvariety of $\bA_g$. 

Regarding the relation between the rank of $M^{\textrm{rel}}_\sigma$ and the geometric properties of $\sigma$ we prove the following:
\begin{cor}\label{mainThm2}
    Under the same hypotheses as in \Cref{mainThm1}, if the (linear) action of the monodromy group $\textrm{Mon}(\mathcal A)$ is irreducible,  then  for any section $\sigma\colon S\to \mathcal A$  ramified over $U_p$, the relative monodromy group $M_\sigma^{\textrm{rel}}$ is isomorphic to $\mathbb Z^{2g}$.
\end{cor}

Also \Cref{mainThm2} can be extended for all non-torsion sections under some additional hypothesis. For instance the action of $\textrm{Mon}(\cA)$ is irreducible when $T$ is not contained in any proper special subvariety of $\mathbb A_g$ (see \cite[10.2.6 Lemma]{ACZ}) or when the abelian scheme is simple. Hence our \Cref{mainThm2} holds for all sections, which are ramified outside a Zariski closed subset, of a simple abelian scheme without fixed part over a curve $S$. Our results suggest the following conjecture:
\begin{conj}\label{conj:main_conj}
Let  $\phi\colon \cA\to S$ be an abelian scheme of relative dimension $g$  without fixed part. The rank of $M^{\textrm{rel}}_\sigma$ is $2g'$ where $g'\le g$ is the relative dimension of the minimal group subscheme of $\cA\to S$ containing the image of the section $\sigma$.
\end{conj}

\Cref{mainThm2} proves \Cref{conj:main_conj} when the modular map is finite and surjective onto $\bA_g$. Notice that such a hypothesis prevents our abelian scheme from having any non-trivial proper group-subscheme. For the same reason, in this case the scheme automatically is without fixed part. Moreover, the second author of the present paper in \cite{T1} proved \Cref{conj:main_conj} in the case of products of two elliptic schemes over a curve.

After the initial version of this paper (containing somewhat weaker results) was circulated, Y. André kindly sent us a preprint~\cite{And_mon}, which addresses certain special cases of~\Cref{conj:main_conj} when \( g' = g \), overlapping with parts of our work. 
His approach is fundamentally Hodge-theoretic, and his main hypothesis on the abelian scheme is the non-thinness of \(\mathrm{Mon}(\mathcal{A})\), in contrast to our condition on the modular map \( p \). A more detailed comparison is provided in~\Cref{sec:3.5}.

\paragraph{Application I: properties of the Betti map.}
Our theorem has immediate implications on the arithmetic properties of the sections of abelian schemes, in particular on some properties of the Betti map. On one hand periods and abelian logarithms in general cannot be globally defined, but on the other hand it is also true that they can be always be globally defined after going to the universal cover $\widetilde{S}$ of the base $S(\bC)$, so the Betti map is also globally defined on $  \widetilde{S}$. Let us denote by $S_* \to S(\bC)$ the minimal topological cover where periods are globally defined and by $S_\sigma \to S_*$ the minimal topological cover where abelian logarithm of $\sigma$ is globally defined: then, $S_\sigma \to S_*$ is also the minimal topological cover of $S_*$ where the Betti map is globally defined and we have the following situation
\[
\widetilde{S} \to S_\sigma \to S_* \to S(\bC)\,.
\]
The map $S_* \to S(\bC)$ is a Galois cover whose Galois group is $\textrm{Mon}(\cA,s)$, while the map $S_\sigma \to S(\bC)$ is a Galois cover with Galois group $M_\sigma$. Our main theorem determines properties of the Galois group $M_\sigma^\textrm{rel}$ of the Galois cover $S_\sigma \to S_*$, giving precise information about the global definition of the Betti map which is useful to know for several applications.

For instance, it is straightforward to see that if $\sigma$ is a torsion section then the Betti map of $\sigma$ is constant and thus globally defined on the base. The converse statement is also true but more difficult to prove, it is widely known as Manin's kernel theorem. As a consequence of \Cref{mainThm1}, we obtain a new proof of a stronger version of Manin's kernel theorem. We call it ``strong version'' because in its statement we just assume the Betti coordinates to be globally defined instead of rational constant a priori.

\paragraph{Application II: the functional transcendence step in the Pila-Zannier method.} The Betti map is used in the work \cite{PZ} of  Pila and Zannier together with a result of Pila and Wilkie in transcendental Diophantine geometry (see \cite{PiWi}) to give a new proof of the Manin-Mumford conjecture. Such a strategy is referred as ``the Pila-Zannier method''; in the case of abelian schemes over $\overline{\mathbb Q}$ it works in the following way: the goal is to have a control on the distribution of torsion values of $\sigma$ (i.e. elements of $\sigma^{-1}(\cA_{\tor})$) with bounded heights. First one gives  lower bounds on the number of such torsion values using Galois conjugates and the height inequality of Dimitrov-Gao-Habegger. On the other hand the Betti map $\beta_\sigma$ transforms the torsion values in rational points of a definable set in the $o$-minimal structure $\mathbb R_{\an,\exp}$. At this point a ``Pila-Wilkie type'' result  gives a  control on the rational points on the transcendental part of such definable set. Still we have no information on the rational points in the algebraic part of the definable set, but here the crucial point is to use some functional transcendence theorem of ``Ax-Schanuel type" to control the algebraic part of the definable set. Let us take a closer look at the last step involving the functional transcendence arguments: a classical way to tackle this step is by using a result of Andr\'e (see \cite[Theorem 3]{André}). It says that the coordinates of the abelian logarithm $\log_\sigma$ are pairwise algebraically independent over the extension of $\mathbb C(S)$ generated by the coordinates of the period map $\mathfrak P$. Such result has often very strong implications, for instance in \cite{DT}, thanks to such result, it is  shown that the algebraic part of the definable set constructed with the Pila-Zannier method is empty.

Such type of functional transcendence results are related to the study of the relative monodromy. In fact Andr\'e's result can be interpreted in terms of differential Galois theory. More precisely, considering directional derivatives along a tangent vector field $\partial$ with respect to the Gauss-Manin connection, the transcendence degree of the extension
\[
\bC(S)(\fP,\partial\fP)(\log_\sigma, \partial\log_\sigma)/\bC(S)(\fP,\partial\fP)
\]
is equal to the dimension of the kernel of the homomorphism from the differential Galois group of $\bC(S)(\fP,\partial\fP)(\log_\sigma, \partial\log_\sigma)/\bC(S)$ to the differential Galois group of $\bC(S)(\fP,\partial\fP)/\bC(S)$. Recall that by general theory, the differential Galois group of a Picard-Vessiot extension is exactly the Zariski closure of the corresponding monodromy group. In \cite{CZ} it is pointed out that André's theorem gives no information about the relative monodromy group $M_\sigma^\textrm{rel}$. Thus, \Cref{mainThm2} can be interpreted as a strengthening of André's theorem (under our hypothesis). We shall in fact prove how André's result is a consequence of \Cref{mainThm2}, and  this in turns leads to a new proof of the algebraic independence of the coordinates of $\log_\sigma$ with respect to periods.

\paragraph{Further developments.} After the works of Andr\'e, Corvaja, Masser and Zannier (see for instance \cite{ACZ}, \cite{CMZ}, \cite{CZ}, \cite{Zbook}) the Betti map became a standard tool in Diophantine geometry. Remarkably,  Dimitrov, Gao and Habbegger in \cite {DGH} used the Betti map and a novel height inequality (then generalized by Yuan and Zhang in \cite{YZ}) to prove a uniform version of the Mordell-Lang conjecture. The Pila-Zannier method was in turns employed by Gao and Habbegger in combination with  new ideas involving some ``degeneracy loci'', to  solve the  relative Manin-Mumford conjecture (see \cite{GH}). Other very recent applications of the Betti map can be found in the works of Xie and Yuan in \cite{XY} towards the geometric Bombieri-Lang conjecture. 

Our hope  for future developments is twofold. First of all, the main results of this paper could be applied for the solution of some ``special points problems'' when in the functional transcendence step of the Pila-Zannier method  André's theorem fails to give a control on rational points of the algebraic part of the constructed definable set. It is not clear yet if such cases exist and if they are interesting. Secondly, it would be nice to remove the hypothesis of maximal variation in moduli and also to obtain the result for all non-torsion sections in \Cref{mainThm1}.

\paragraph{Strategy of the proofs.}

Our proofs develop the ideas introduced in \cite{CZ}, interpreting the obstruction to defining a global logarithm as a Galois cohomology class relative to the monodromy action of $\pi_1(S(\mathbb{C}))$ on $\mathbb{Z}^{2g}$. The key innovation involves applying the Lefschetz hyperplane theorem for quasi-projective varieties (see \cite[Part II, 1.2, Theorem]{GorMac}) to reduce the problem to generic curves in the base variety. We then utilize properties of the trace operator with respect to the pullback of sections (in the case of finite flat morphisms). Unlike the approach in \cite{CZ}, our method does not require the construction of abelian schemes with finite Mordell-Weil groups.

For the proof of \Cref{mainThm1} we proceed in the following way:  we construct an intermediate abelian scheme $\cA' \to S'$ which factorizes the universal diagram induced by the modular map $p:S \to T$. This construction is carried out by means of topological techniques in such a way that $\cA' \to S'$ is a finite unramified base change (after normalization) of the universal family. Then we find a suitable curve inside $S(\bC)$ which preserves the complete monodromy action of the fundamental group of $S(\bC)$ on periods and logarithms. The curve is obtained after iterating the aforementioned Lefschetz theorem and turns out to be a Stein space. This last property allows us to construct holomorphic sections of the unramified family which come from a Galois cohomology class describing the obstruction of the global logarithms. The proof then follows: we assume the triviality of $M_\sigma^\textrm{rel}$, then by using the previous apparatus and the properties of trace and pull-back operators, we deduce that $\sigma$ should be unramified. This proves the non-triviality of $M_\sigma^{\textrm{rel}}$ for ramified sections (over $U_p$).

Once that we have addressed the non-triviality of $M^{\textrm{rel}}_\sigma$, the corollaries follow immediately from the extra assumptions.

\paragraph{Acknowledgmentss.} 
The authors are deeply grateful to \emph{P. Corvaja} and \emph{U. Zannier} for their support and numerous insights. Moreover, they thank \emph{Y. Andr\'e} for the interesting discussions on the topic.
The second author is supported by PRIN 2022, CUP F53D23002740006. Moreover, he would like to thank the Institute for Theoretical Sciences of Westlake University for their hospitality while he was completing the paper.

\section{Generalities on abelian schemes}\label{section2}

\subsection{Ramification and torsion}\label{sec:tor_ram}

We discuss here the notion of unramified section formally given in \Cref{{def:ram_sect}}, providing examples and properties. We prove a characterization useful for constructing unramified sections, and we show that torsion sections are always unramified. We refer to the notation introduced at the beginning and to the diagram 
\begin{equation}\label{eq:relNormalization}
	\begin{tikzcd}
		\cA \arrow{d}{\phi} \arrow{r} & \cA' \arrow{d}{\phi'}\arrow{r} & \fA_{g,T^\nu} \arrow{d}\arrow{r} & \fA_{g,T}\arrow{d}\\
		S \arrow{r}{q} \arrow[bend right, swap]{rr}{p^\nu}& S^{\ur} \arrow{r}{q'} & T^\nu \arrow{r}{\nu} & T,
	\end{tikzcd}
\end{equation}
where $q'$ is unramified, $q$ is dominant, $S^{\ur}$ is normal and $\nu$ is the normalization map.

In general, the normalization map $\nu \colon T^\nu \to T$ is not unramified, and understanding its ramification is an interesting problem. For instance, the normalization of a nodal curve is unramified but not étale, while the normalization of a cuspidal curve is ramified.  On the other hand, if we start with a normal irreducible variety, the relative normalizations within finite field extensions of the function field yield natural examples of unramified morphisms. In fact, there is a strong classification result: étale covers of a normal integral variety, such as $T^\nu$, correspond exactly to the relative normalizations of $T^\nu$. A more general version of this statement, which includes non-integral schemes, can be found in \cite[Exposé I, Corollaire~10.2 and Corollaire~10.3]{SGA1}.

We aim to establish criteria for identifying unramified sections $\sigma \colon S \to \mathcal{A}$. At first glance, one might expect that unramified sections should be defined as those induced by unramified covers of $T$, rather than $T^\nu$. However, since $T$ may fail to be normal, this definition proves to be ill-behaved, as previously discussed.   Nevertheless, we observe that due to the ``cancellation property'' of unramified morphisms (see \cite[TAG 02GG]{stacks-project}), sections arising from unramified covers of $T$ remain unramified under our current definition. The converse, however, does not hold in general, as the normalization map $\nu \colon T^\nu \to T$ may itself be ramified.

We remark that the normality condition is essential in order to have the correspondence between the notion of topological covers and finite étale covers. In fact while finite étale morphisms induce topological covers of the corresponding complex spaces, we need the normality condition to conclude that an unramified (finite) topological cover induces a finite étale map. This implies the first simple criterion to construct unramified sections:
\begin{prop}\label{prop:unramifiedMiddle}
    Assume the modular map $p:S \to T$ to be finite. Let $\sigma:S \to \cA$ be a section. Let $\cA''/S''$ an abelian scheme with a cartesian diagram
    \[
    \begin{tikzcd}
        \cA \arrow{r} \arrow{d}{\phi} & \cA'' \arrow{r} \arrow{d} & \fA_{g,T^\nu} \arrow{d}\\
        S \arrow[bend left]{u}{\sigma} \arrow{r}{g} \arrow[bend right, swap]{rr}{p^\nu} & S'' \arrow{r}{f} & T^\nu,
    \end{tikzcd}
    \]
    where $f$ is an unramified morphism. If $\sigma$ is pullback of a section of the abelian scheme $\cA''/S''$ by $g$, then $\sigma$ is an unramified section.
\end{prop}

\begin{proof}
    By \cite[Exposé I, Corollaire~9.11]{SGA1}, the morphism $f$ is étale, and $S''$ is normal and irreducible. By \cite[Exposé I, Corollaire~10.2 and Corollaire~10.3]{SGA1}, it follows that $\sigma$ is the pullback of a section of the abelian scheme $\cA'/S^{\ur}$, thus it is unramified.\\
\end{proof}

Let $A$ and $A'$ be the generic fibres of $\mathcal A$ and $\mathcal A'$ respectively, since the induced map $ A\to A'$ is dominant, any section $\sigma:S\to\mathcal A$ induces a $\mathbb C(S)$ point of $A'$ that we denote by $\sigma'$. The closure of $\sigma'$ in $\mathcal A'$ is denoted by $\Sigma'$ and clearly $\mathbb C(\Sigma')$ embeds into $\mathbb C(S)$. Moreover recall that $\phi'\colon \Sigma'\to S^{\ur}$ is surjective.

\begin{prop}\label{prop:tor_ram}
Keep the notations fixed above and in \Cref{eq:relNormalization}.  A section $\sigma\colon S\to\mathcal A$ is unramified 
if and only if there exists a locally closed subscheme $\mathcal B/S^{\ur}$  of $\cA'/S^{\ur}$ with the following property: the restriction $\phi': \mathcal B \to S^{\ur}$ is finite unramified. If $\sigma: S \to \cA$ is a section such that $\Sigma'\subseteq\mathcal B$, then $\sigma$ is unramified.
\end{prop}

\proof
  Assume that $\sigma$ is unramified, so $\sigma=q^\ast\sigma'$ where $\sigma'\colon S^{\ur}\to \mathcal A'$ is  a section of $\phi'$. Then $\Sigma'=\sigma'(S^{\ur})$ and it is enough to take $\mathcal B:=\Sigma'$. Vice versa assume the existence of $\mathcal B$; by \cite[Tag 02IK]{stacks-project}  $\mathcal B$ is a Noetherian scheme (so a Noetherian topological space) then it has a finite number of irreducible components. This implies that all the connected components are closed and open. Therefore, the restriction of the morphism $\phi'\colon\mathcal B \to S^{\ur}$ to its connected component $\mathcal B'$ that contains $ \Sigma'$  is finite and unramified.  By \cite[Tag 0GS9]{stacks-project} (here we use the hypothesis that  $S^{\ur}$ is normal) we conclude that $\phi':\mathcal B'\to S^{\ur}$ is finite étale  and $\mathcal B'$ is irreducible. So $\mathcal B'=\Sigma'$ because they have the same dimension. We have shown that $\Sigma'$ is a connected component of $\mathcal B$. Now since $\mathbb C(\Sigma')\subseteq\mathbb C(S)$,  and $\phi':\Sigma'\to S^{\ur}$ is in particular unramified it follows that $\mathbb C(S^{\ur})=\mathbb C(\Sigma')$, this because $\mathbb C(S^{\ur})$ is the maximal unramified extension of $\mathbb C(T^{\nu})$ in $\mathbb C(S)$. It follows that  the morphism $\phi'\colon\Sigma'\to S^{\ur}$ is finite étale of degree $1$, hence an isomorphism. By \cite[Tag 024U]{stacks-project} there exists a section $\widetilde{\sigma}\colon S^{\ur}\to\mathcal B$ such that $\widetilde{\sigma}(S^{\ur})=\Sigma'$. In particular $\widetilde{\sigma}$ is a section of $\phi':\mathcal A'\to S^{\ur}$ and $q^\ast\widetilde{\sigma}=\sigma$.
\endproof

\begin{cor}\label{cor:divisible_ram}
Let $m \in \bZ_{>0}$ be a positive integer. A section $\sigma:S \to \cA$ is unramified if and only if $m\sigma$ is unramified. In particular, any torsion section is unramified
\end{cor}
\proof
If $\sigma$ is unramified then $\sigma=q^\ast\sigma'$ for $\sigma'\colon S^{\ur}\to\cA'$, so $m\sigma=q^\ast(m\sigma')$. Vice versa, denote by $f:\cA \to \cA'$ the upper left horizontal map in \Cref{eq:relNormalization}. Since $m\sigma$ is unramified, there exists a section $\psi:S^{\ur} \to \cA'$ such that $m\sigma=q^\ast\psi$. This condition implies that $[m]\circ f\circ\sigma=\psi\circ q$. Define $\Sigma':=f\circ\sigma(S)\subseteq \cA'$ and observe this corresponds to take the closure of the image of $\sigma$ in the generic fiber as above. Thus, $\Sigma'$ is irreducible and clearly $\bC(\Sigma')$ embeds into $\bC(S)$. The multiplication map $[m]:\cA' \to \cA'$ is  finite étale  by \cite[V, Proposition 20.7]{Cor_Sil}. Let us consider the locally closed subscheme $\mathcal B/S^{\ur}$ satisfying the following cartesian diagram:
\[
\begin{tikzcd}
    \mathcal B \arrow{r} \arrow{d}{\phi'} & \cA' \arrow{d}{[m]}\\
    S^{\ur} \arrow{r}{\psi} & \cA'.
\end{tikzcd}
\]
Observe that the map $\mathcal B \to S^{\ur}$ is étale, because it is base change of an étale morphism. Since $\Sigma' \subseteq \mathcal B$, the conclusion follows by \Cref{prop:tor_ram}.
\endproof

We now provide an example of unramified non-torsion section.

\begin{expl}\label{expl:tor_ram}
    Let us denote by $\mathcal L \to S_0$ the Legendre elliptic scheme, where $S_0:=\mathbb P_1 \setminus \{0,1,\infty\}$. Consider the product
    \[
    (\phi_1,\phi_2):\mathcal L \times \mathcal L \to S_0\times S_0,
    \]
    where we use the coordinates $(x,y)$ in $S_0\times S_0$. We let $X$ be the line in $S_0^2$ defined by $x+y= 2$, so this line is isomorphic under the first projection to $B:=\mathbb P_1 \setminus \{0,1,2,\infty\}$. By projection from $X$ to $B$, we may define a scheme $\phi: \cA \to B$ over $B$ whose fibers are of the type $\phi^{-1}(\lambda) = \mathcal L_\lambda \times \mathcal L_{2-\lambda}$.


    The abelian scheme $\phi:\cA \to B$ may be interpreted as the fiber product of two elliptic schemes $\cE_1 \times_B \cE_2 \to B$, where $\cE_1 \to B$ is the Legendre scheme associating $\mathcal L_\lambda \to \lambda$ and $\cE_2 \to B$ is the Legendre scheme associating $\mathcal L_\lambda \to 2-\lambda$. Define the multisection:
    \[
    \tau(\lambda):=\left(\left(2,\sqrt{2(2-\lambda)}\right), \left(2, \sqrt{2\lambda}\right)\right).
    \]
    It becomes a non-torsion section on the base change induced by the degree $4$ field extension $\bC(B) \subseteq \bC(B)(\psi,\mu)$ where $\psi^2+\mu^2=2$ and $\mu^2=\lambda$. Such a field extension induces an unramified morphism onto $B$ and therefore this gives rise to an unramified (finite flat) modular map. This implies that $\tau$ is an unramified section.
\end{expl}

\subsection{Periods and abelian logarithms}\label{perAbLog}

Let us consider an abelian scheme $\phi:\cA \to S$ with a zero-section as above. Each fiber $\cA_s(\bC)$ is analytically isomorphic  to  a complex torus $\bC^g/\Lambda_s$ and for any subset $W\subseteq S(\mathbb C)$ we denote $\Lambda_{W}:=\bigsqcup_{s\in W} \Lambda_s$. The space $\Lie(\cA):=\bigsqcup_{s\in S(\mathbb C)}\Lie(\mathcal A_s(\mathbb C))$ has a natural structure of $g$-dimensional holomorphic vector bundle $f\colon\Lie(\cA)\to S(\bC)$ (it is actually  a complex Lie algebra bundle). By using the fiberwise exponential maps one can define a global  map $\exp\colon \Lie(\cA)\to \cA(\mathbb C)$. Let $\Theta_0\subset\cA(\mathbb C)$ be the image of the zero section of the abelian scheme, then  obviously
\begin{equation}\label{periodSheafEq}
    \exp^{-1}(\Theta_0)=\Lambda_S.
\end{equation}
Clearly $S(\mathbb C)$ can be covered by open simply connected subsets where the holomorphic vector bundle $f\colon\Lie(\cA)\to S(\bC)$ trivializes. Let $U\subseteq S(\bC)$ be any of such subsets and consider the induced holomorphic map $f\colon \Lambda_U\to U$; it is actually a fiber bundle with structure group $\textnormal{GL}(n,\bZ)$.  Since $U$ is simply connected, by \cite[Lemma 4.7]{DK} we conclude that $f\colon \Lambda_U\to U$ is a  topologically trivial fiber bundle. Thus we can find $2g$ continuous sections of $f$:
\begin{equation}\label{periods}
    \mathcal \omega_i:U\to\Lambda_{U}\,,\quad i=1,\ldots 2g
\end{equation}
such that $\{\omega_1(s),\ldots,\omega_{2g}(s)\}$ is a basis of periods for $\Lambda_s$ for any $s\in U$. Since $\Lambda_U \subseteq \Lie(\cA)_{|U}$, we can put periods into the following commutative diagram:
$$
\begin{tikzcd}
     & \Lie(\cA)_{|U} \arrow{d}{\exp_{|U}}\\
    S(\bC) \supset U \arrow{r}{\sigma_0{|U}} \arrow{ur}{\omega_i} & \cA_{|U},
\end{tikzcd}
$$
where $\sigma_0$ is the zero section. Since $\sigma_0$ is holomorphic and $\exp$ is a local biolomorphism, then the period functions defined in \Cref{periods} are holomorphic. The map
\begin{equation}\label{periodMap}
    \fP=(\omega_1,\ldots,\omega_{2g})
\end{equation}
is called a \emph{period map}; roughly speaking it selects a $\bZ$-basis for $\Lambda_s$ which varies holomorphically for $s\in U$. 

Let us consider now a non-torsion section $\sigma: S\to \cA$ of the abelian scheme. The set $U\subseteq S(\bC)$ is simply connected, therefore we can choose a holomorphic lifting $\ell_\sigma:U\to\Lie(\cA)$ of the restriction $\sigma_{|U}$; $\ell_\sigma$ is often called an \emph{abelian logarithm}. Thus for any $s\in U$ we can write uniquely
\begin{equation}\label{log}
\ell_\sigma(s)=\beta_1(s)\omega_1(s)+\ldots+\beta_{2g}\omega_{2g}(s)
\end{equation}
where $\beta_i: U\to\mathbb R$ is a real analytic function for $i=1,\ldots, 2g$. The map 
$\beta_\sigma:U\to\mathbb R^{2g}$
defined as $\beta_\sigma=(\beta_1,\ldots,\beta_{2g})$ is called the \emph{Betti map associated to the section $\sigma$}, whereas the $\beta_i$'s are the \emph{Betti coordinates}. Observe that the Betti map depends both on the choices of the period map $\mathcal P$ and of the abelian logarithm $\ell_\sigma$, but this is irrelevant for our aims. 

As already mentioned, in this paper we are going to make also use of the sheaf of holomorphic sections $\Sigma^{\an}$. Let $\cLie(\cA)$ denote the locally free sheaf on $S(\mathbb C)$ associated to the vector bundle $\textrm{Lie}(\cA)\to S(\mathbb C)$. Then we have a morphism of sheaves $\psi\colon \cLie(\cA) \to \Sigma^{\an}$ defined in the following way on any open set $U\subseteq S(\mathbb C)$:
\begin{eqnarray*}
\psi_U\colon \cLie(\cA)(U)& \to& \Sigma^{\an}(U)\\
t &\mapsto& \exp \circ\, t.
\end{eqnarray*}
Observe that $\psi$ is surjective  because of the local  existence of abelian logarithms. Moreover the sheaf of periods $\Lambda_S$ is exactly the kernel of $\psi$, therefore we obtain the following short exact sequence of sheaves of abelian groups on $S$:
\begin{equation}\label{exSeqSch}
    0 \rightarrow \Lambda_S \rightarrow \cLie(\cA) \rightarrow \Sigma^{\an} \rightarrow 0.
\end{equation}

\section{Monodromy groups}
\subsection{Monodromy of abelian schemes}
Let $\cF$ be a sheaf of abelian groups over a topological space $X$ and denote by $\Gamma$ the functor of ``global sections'' which to $\cF$ associates $\Gamma(X,\cF) = \cF(X)$, with values in the category of abelian groups. Since the category of sheaves of abelian groups has sufficiently many injective objects, then the derived functors $R^k\Gamma$ do exist. They are generally written
\[
R^k\Gamma(\cF)=:H^k(X,\cF).
\]
Let $X$ be a locally contractible topological space and denote by $\underline{\bZ}$ the constant sheaf of abelian groups over $X$ with constant stalk $\bZ$. Then we have a canonical isomorphism
\[
H^q_{\textrm{sing}}(X,\bZ) \cong H^q(X,\underline{\bZ}),
\]
where we are considering the cohomology of $X$ with coefficients in the constant sheaf of stalk $\underline{\bZ}$ on the right, and the singular cohomology with coefficients in $\bZ$ on the left. The same result holds with $\bZ$ replaced by any commutative ring $G$.

Now, let $\phi:\cA \to S$ be an abelian scheme as defined above. We want to define a sheaf on the base $S$ containing information on periods of the abelian schemes, to this end let us consider the direct image functor $\phi_*$. Since $\phi_*$ is left exact and the category of sheaves of abelian groups has enough injectivities, the right derived functors of $\phi_*$ are well defined from the category of sheaves on $\cA$ to the category of sheaves on $S$: they are called \emph{higher direct image functors} and will be denoted by $R^k\phi_*$. For any $k \ge 0$, define the sheaf
\[
\cP_S^k:=R^k\phi_*\underline{\bZ}.
\]
By \cite[Proposition 8.1]{Hart}, this is the sheaf on $S$ associated to the presheaf
\[
U \mapsto H^k(\phi^{-1}(U),\underline{\bZ}_{|\phi^{-1}(U)}).
\]

Since $\phi:\cA \to S$ is a smooth surjective morphism of algebraic varieties over $\bC$, by \cite[Theorem 10.4]{Hart} then $\phi$ is a submersion. We can apply Ehresmann's Lemma to conclude that the proper submersion $\phi$ is a $C^\infty$-fiber bundle.

Since $S(\bC)$ is locally contractible, then for sufficiently small $U \subseteq S(\bC)$, the open sets $\cA_U:=\phi^{-1}(U)$ have the same homotopy type as the fibre $\cA_s(\bC)$ with $s \in U$. Using the invariance under homotopy, i.e. the fact that if $U$ is a contractible space then $H^k(\cA_s(\bC) \times U,\underline{\bZ}) = H^k(\cA_s(\bC), \underline{\bZ})$ for all $k \ge 0$, we deduce that the sheaves $\cP_S^k=R^k\phi_*\underline{\bZ}$ are locally constant sheaves (or equivalenty \emph{local systems}) on $S$ with stalk
\[
(\cP_S^k)_s=(R^k\phi_*\underline{\bZ})_s = H^k(\cA_s(\bC),\underline{\bZ}).
\]

Now, the fundamental group $\pi_1(S(\bC),s)$ acts via linear transformations on $(\cP_S^k)_s$: this gives rise to a monodromy representation. We are going to give a more detailed description of the monodromy representation but the rough idea is the following: pick a loop $\gamma(t)$ at $s$ and use a trivialization of the bundle along the loop to move vectors in $(\cP_S^k)_s$ along $(\cP^k_S)_{\gamma(t)}$, back to $(\cP_S^k)_s$. Actually this is a more general construction regarding monodromy representations associated to local systems, but in our case the monodromy action can be induced by homeomorphisms of the fiber. In order to be more precise, consider $\gamma \in \pi_1(S(\bC),s)$ represented by a loop $\gamma:[0,1] \to S(\bC)$ based at $s$. Consider the fibration $\cA_\gamma$ defined as the fiber product
$$\begin{tikzcd}
    \cA_\gamma \arrow{d}[swap]{\phi_\gamma} \arrow{r} & \cA(\bC) \arrow{d}{\phi}\\
    \left[0,1\right] \arrow{r}{\gamma} & S(\bC).
\end{tikzcd}$$
By the compactness of $[0,1]$, there exist real numbers $0\le \epsilon_i < \epsilon_{i+1}, 1\le i \le N$, such that $\epsilon_1 =0,\epsilon_N=1$, and $\phi_\gamma$ trivialises on $[\epsilon_i,\epsilon_{i+1}]$. By gluing local trivialisations above segments $[\epsilon_i,\epsilon_{i+1}]$, we can trivialise the fibration $\phi_\gamma: \cA_\gamma \to [0,1]$. From such a trivialization
\[
\cA_\gamma \cong \phi_\gamma^{-1}(0) \times [0,1],
\]
we deduce a homeomorphism $\psi: \phi_\gamma^{-1}(1) \cong \phi_\gamma^{-1}(0)$. These two spaces are canonically homeomorphic to the fibre $\cA_s(\bC)$.

The homeomorphism $\psi$ induces a group automorphism $\psi^k$ of $H^k(\cA_s(\bC),\underline{\bZ})$, for any $k\ge 0$. Therefore, we obtain the monodromy representations
\begin{equation}\label{eq:mon_repr}
\rho^k: \pi_1(S(\bC),s) \to \textrm{Aut}\, H^k(\cA_s(\bC),\underline{\bZ})
\end{equation}
defined by
\[
\rho^k(\gamma)(\eta) := \psi^k\eta, \qquad \eta \in H^k(\cA_s(\bC),\underline{\bZ}). 
\]
Thus, thanks to the fact the local systems we are considering is obtained using a fibration, the monodromy representation is in fact induced by homeomorphisms of the fibre $\cA_s(\bC)$ onto itself. In particular, it follows that the monodromy representation on the cohomology of the fibre $H^\bullet(\cA_s(\bC),\underline{\bZ})$ is compatible with the cup-product on $H^\bullet(\cA_s(\bC),\underline{\bZ})$, in the sense that each $\rho(\gamma) \in \textrm{Aut}\,H^\bullet(\cA_s(\bC),\underline{\bZ})$ is a ring automorphism for the ring structure given by the cup-product. 


Recall that by Poincar\'e duality we have isomorphisms $H^k(\cA_s(\bC),\underline{\bZ}) \cong H_{2g-k}(\cA_s(\bC),\bZ)$. Moreover, by the K\"unneth formula the groups $H_k(\cA_s(\bC),\bZ)$ and $H^k(\cA_s(\bC),\underline{\bZ})$ are free abelian groups of finite rank $\binom{2g}{k}$ for all $k = 1, \ldots, 2g$. Therefore, we can look at $\cP_S^{2g-1}$ as a local system with constant stalk $\bZ^{2g}$ to be identified with singular homology; we will simply write $\cP$ instead of $\cP_S^{2g-1}$. In other words, from now on we fix $k=2g-1$ and we only deal with the monodromy action induced on the first homology group identifying $\textrm{Aut}\,H_1(\cA_s(\bC),\bZ) \cong \textrm{GL}_{2g}(\bZ)$:
\[
\rho:\pi_1(S(\bC),s) \to \textrm{GL}_{2g}(\bZ).
\]
Furthermore, we have a non-degenerate intersection form
\[
\langle\cdot, \cdot\rangle: H_1(\cA_s(\bC),\bZ) \times H_{2g-1}(\cA_s(\bC),\bZ) \to \bZ.
\]
Any functional $H_{2g-1}(\cA_s(\bC),\bZ) \to \bZ$ is an intersection index of some homology class from $H_1(\cA_s(\bC),\bZ)$ and, if a class $A \in H_1(\cA_s(\bC),\bZ)$ is such that $\langle A,B \rangle = 0$ for any $B \in H_{2g-1}(\cA_s(\bC),\bZ)$, then $A$ vanishes as element of $H_1(\cA_s(\bC),\bZ)$. The abelian variety $\cA_s(\bC)$ has a polarization inducing an alternating form $E$ which takes integer values on the lattice $\Lambda_s$. Identifying $\Lambda_s=H_1(\cA_s(\bC),\bZ)$ and fixing a symplectic basis $\omega_1, \ldots, \omega_{2g}$ of $H_1(\cA_s(\bC),\bZ)$ for $E$ we obtain functionals
\[
\varphi_i:=E(\cdot,\omega_i):H_1(\cA_s(\bC),\bZ) \to \bZ, \qquad \textrm{for } i =1, \ldots, 2g.
\]
For any functional $\varphi_i$ there exists a $(2g-1)$-cycle $\eta_i \in H_{2g-1}(\cA_s(\bC),\bZ)$ such that
\[
\varphi_i(\omega_j)=\langle \eta_i,\omega_j \rangle.
\]
Since the intersection form is dual to the cup-product and since each $\rho(\gamma) \in \textrm{GL}_{2g}(\bZ)$ is a ring automorphism for the ring structure given by the cup-product, then $\rho(\gamma)$ has to preserve the intersection matrix induced by $E$, i.e.
\[
P=
\begin{pmatrix}
    0 & I_g\\
    -I_g & 0
\end{pmatrix}.
\]
In other words the image of the representation $\rho$ is contained in the following group
\[
\textrm{Sp}_{2g}(\bZ)=\{M \in \textrm{GL}_{2g}(\bZ) \mid M^\top PM=P\}.
\]
Thus, monodromy representation of periods introduced in \Cref{eq:mon_repr} can be actually thought as a map:
\begin{equation}\label{monPer}
    \rho: \pi_1(S(\bC),s) \to \textrm{Sp}_{2g}(\bZ).
\end{equation}

\begin{defn}
We define the \emph{monodromy group} of $\cA \to S$ at $s$ as $\textrm{Mon}(\cA,s):= \rho(\pi_1(S(\bC),s))$.
\end{defn}
Since $S(\bC)$ is path connected, all the groups $\textrm{Mon}(\cA,s)$ are conjugate when we vary the base point $s \in S(\bC)$: in other words, the monodromy group is defined up to an inner automorphism of the group $\textrm{GL}_{2g}(\bZ)$. Thus, fix once and for all a base point $s \in S(\bC)$ and denote the group $\textrm{Mon}(\cA,s)$ simply by $\textrm{Mon}(\cA)$, without writing any dependencies on the base point.

Since we can identify each first homology group of the fibers with the corresponding lattice $\Lambda_s$, then the sheaf $\cP$ can be identified with the sheaf $\Lambda_S$ defined in \Cref{periodSheafEq}; it will be called \emph{period sheaf} and will be denoted simply by $\Lambda$. In more concrete terms, notice that by construction studying the monodromy of the period sheaf corresponds to studying the analytic continuation of period functions (i.e. the coordinates of the period map defined in \Cref{periodMap}) along loops in $S(\bC)$ representing classes in the fundamental group based at $s$.

\subsection{The case of universal families}

For universal families, a more explicit description of the monodromy action on periods is described in \cite[Section 3]{GHDeg}. Let $\bH_g$ denote Siegel's upper half space, i.e., the symmetric matrices in $\textrm{Mat}_{g\times g}(\bC)$ with positive definite imaginary part. We have holomorphic uniformizing maps
\[
u_B: \bH_g \to \bA_g(\bC) \qquad \textrm{and} \qquad u: \bC_g \times \bH_g \to  \fA_g(\bC).
\]
Recall that $\textrm{Sp}_{2g}(\bR)$, the group of real points of the symplectic group, acts on $\bH_g$. Let $x \in \bA_g(\bC)$ and fix $\tau \in \bH_g$ such that $x = u_B(\tau)$. If $1_g$ denotes the $g\times g$ unit matrix, then the columns of $(\tau, 1_g)$ are an $\bR$-basis of $\bC^g$ and we have $\fA_{g,x}(\bC) \cong \bC^g/(\tau \bZ^g + \bZ^g)$. The period lattice basis $(\tau, 1_g)$ allows us to identify $H_1(\fA_{g,x}(\bC),\bZ)$ with $\bZ^{2g}$. Let us consider a loop $\gamma$ in $\bA_g(\bC)$ based at $x$ representing a class of $\pi_1(\bA_g(\bC),x)$. Then a lift $\widetilde{\gamma}$ of $\gamma$ to $\bH_g$ starting at $\tau$ ends at $M\tau \in \bH_g$ for some
\[
M = \begin{pmatrix}
    a & b\\
    c & d
\end{pmatrix}
\in \textrm{Sp}_{2g}(\bZ).
\]
Then $M\tau$ is the period matrix of the abelian variety $\bC^g/(M\tau\bZ^g + \bZ^g)$ which is isomorphic to $\bC^g/(\tau\bZ^g + \bZ^g)$: more precisely, the isomorphism $\bC^g/(\tau\bZ^g+\bZ^g) \to \bC^g/(M\tau\bZ^g+ \bZ^g)$ is induced by the map
\[
\tau u + v \mapsto ((c\tau+d)^\top)^{-1}(\tau u + v) = (M\tau,1_g)(M^\top)^{-1}\begin{pmatrix}
    u\\
    v
\end{pmatrix},
\]
where $u, v \in \bR^g$ are column vectors; in what follows sometimes we will interpret $u,v$ as row vectors and we will consider the transposition of the previous relation. Therefore, the monodromy representation expressed in these coordinates is given by
\begin{equation}\label{columnAction}
    \begin{aligned}
        \rho: \pi_1(\bA_g(\bC)),x) &\to \textrm{Sp}_{2g}(\bZ)\\
        [\gamma] &\mapsto (M^\top)^{-1}.
    \end{aligned}
\end{equation}
Notice that the monodromy action is a right action on period functions, it can be obviously interpreted as a left action by putting $h\cdot (M\tau,1_g):=(M\tau,1_g)\rho(h)$.

The previous considerations about periods of universal families allow us to obtain some results on periods of the abelian scheme $\phi:\cA \to S$: we are able to construct period functions of $\phi: \cA \to S$ by means of periods of the universal family $\pi:\fA_g \to \bA_g$. Let $V \subseteq \bA_g(\bC)$ be a simply connected open set of an open covering of $\bA_g$ where holomorphic period functions do exist for the universal abelian scheme, see \Cref{periods}; denote by $\omega^{\fA_g}_{V,1}, \ldots, \omega^{\fA_g}_{V,2g}$ such local (holomorphic) period functions. Fix a base point $s \in S(\bC)$, not a ramification point for the modular map $p:S \rightarrow \bA_g$, and let $x = p(s) \in \bA_g(\bC)$. In a connected and simply connected neighborhood $U$ of $s$ in $S(\bC)$ such that $p(U)=V$, we can holomorphically define a basis $\omega_{U,1}, \ldots, \omega_{U,2g}$ of the period lattices by the equations
\begin{equation}\label{periodFunctions}
    \omega_{U,i} = \omega^{\fA_g}_{V,i}\circ p.
\end{equation}
Locally on suitable open subsets $U \subset S(\bC)$, this gives a basis for $\Lambda_s$ made up of holomorphic functions $\omega_{U,1}, \ldots, \omega_{U,2g}:U \rightarrow \bC^g$.

\begin{rem}\label{densityMon}
    Recall that the monodromy group of the universal family $\textrm{Mon}(\fA_g)$ is a finite index subgroup of $\textrm{Sp}_{2g}(\bZ)$. When the modular map $p:S \to \bA_g$ is finite, the monodromy group $\textrm{Mon}(\cA)$ of $\cA \to S$ is still a finite index subgroup of $\textrm{Sp}_{2g}(\bZ)$ thanks to \Cref{periodFunctions}; in particular, $\textrm{Mon}(\cA)$ is Zariski-dense in $\textrm{Sp}_{2g}(\bZ)$. Moreover, when the image of the modular map $T$ is not contained in any proper special subvariety of $\bA_g$, the monodromy group $\textrm{Mon}(\cA)$ is still Zariski-dense in $\textrm{Sp}_{2g}(\bZ)$ by \cite[Remark 10.1.2]{ACZ}. 
\end{rem}

\begin{lem}\label{irrSp}
    The natural action of the group $\textrm{Sp}_{2g}(\bZ)$ on $\bQ^{2g}$ is irreducible.
\end{lem}

\begin{proof}
    Assume that there exists a non-trivial subvector space $W$ of $\bQ^{2g}$ which is fixed by the action of $\textrm{Sp}_{2g}(\bZ)$. Since the symplectic group $\textrm{Sp}_{2g}$ is a simple algebraic Lie group defined over $\bQ$, by a theorem of Borel and Harish-Chandra \cite[Theorem 7.8]{BorHC} the group $\textrm{Sp}_{2g}(\bZ)$ is a lattice in $\textrm{Sp}_{2g}(\bR)$. By Borel density theorem (see \cite{Bor}) we have that $\textrm{Sp}_{2g}(\bZ)$ is Zariski dense in $\textrm{Sp}_{2g}(\bR)$, and so it is Zariski dense in $\textrm{Sp}_{2g}(\bQ)$. This implies that $W$ is fixed by the action of $\textrm{Sp}_{2g}(\bQ)$. Since the action of $\textrm{Sp}_{2g}(\bQ)$ on $\bQ^{2g}$ is irreducible (see \cite[Proposition 3.2]{Gro}), we get $W=\bQ^{2g}$ which concludes the proof.\\
\end{proof}

\begin{prop}\label{irredAct}
    If the image of the modular map $T$ is not contained in any proper special subvariety of $\bA_g$, the action of the monodromy group $\textrm{Mon}(\cA)$ on the lattice of periods is irreducible.
\end{prop}

\begin{proof}
    This is an easy consequence of \Cref{densityMon} and \Cref{irrSp}.\\
\end{proof}

\begin{rem}\label{wellDefPeriods}
    The previous considerations yield some conclusion on the global definition of periods on the base $S(\bC)$. We have seen that period functions $\omega_1, \ldots, \omega_{2g}$ can be locally defined on simply connected open sets $U \subseteq S(\bC)$ via \Cref{periodFunctions} for any abelian scheme $\cA\rightarrow S$. These functions may be analytically continued through the whole of $S(\bC)$, but it's impossible to globally define them: they turn out to be multi-valued functions, i.e. they have quite nontrivial monodromy when traveling along closed paths. 

	To be more precise, since the group $\textrm{Mon}(\cA)$ is non-trivial then the functions $\omega_1, \ldots, \omega_{2g}$ cannot be all defined continuously on the whole of $S(\bC)$. Moreover, when the action of $\textrm{Mon}(\cA)$ on the lattice of periods is irreducible, neither one of $\omega_1, \ldots, \omega_{2g}$ nor any single non-zero element of the period lattice can be defined on the whole of $S(\bC)$. This last thing is also proved in \cite[Lemma 5.6]{GH2} for any abelian scheme over a curve without using irriducibility properties of the monodromy action.
\end{rem} 

\subsection{Monodromy groups of sections}

Let us consider the exponential map $\exp: \Lie(\cA) \to \cA(\mathbb C)$ defined in \Cref{perAbLog}. Notice that it is a topological cover. Thus, it induces a monodromy action of the fundamental group $\pi_1(\cA(\bC),p)$ on the fiber $\exp^{-1}(p)$, where we assume $\phi(p)=s$. Let $\sigma: S \to \cA$ be a section of the abelian scheme, we obtain a monodromy action of $\pi_1(S(\bC),s)$ on the fiber $\exp^{-1}(\sigma(s))$.

Recalling the definition of logarithm given in \Cref{log}, the set $\exp^{-1}(\sigma(s))$ is the set of all possible determinations of the abelian logarithm of $\sigma$ at the point $s$. Every two such determinations differ each other by an element of the lattice $\Lambda_s$, therefore, after fixing a branch $\log_\sigma(s)\in\exp^{-1}(\sigma(s)) $ of the logarithm at $s$, we get the monodromy action
\begin{equation*}
    \rho': \pi_1(S(\bC),s) \to \Aut(\mathbb Z^{2g})\,.   
\end{equation*}
The representation $\rho'$  induces a map (which is not a group homomorphism):
\begin{equation}\label{monLog}
\begin{aligned}
    c: \pi_1(S(\bC),s) &\to \bZ^{2g}\\
    h &\mapsto \rho'(h)(\log_\sigma(s))\,.
\end{aligned}
\end{equation}
The function $c$ can be described explicitly. In fact, any element $h \in \pi_1(S(\bC),s)$ acts by analytic continuation on $\log_\sigma(s)$ in the following way:
\begin{equation}\label{actionLog}
	h\cdot\log_\sigma=\log_\sigma + u_1^h\omega_1 + \cdots + u_{2g}^h\omega_{2g},
\end{equation}
where $u_1^h,\ldots, u_{2g}^h \in \bZ$. Then $c(h)=(u^h_1, \ldots, u^h_{2g})^\top$. Again, since $S(\mathbb C)$ is path connected we can fix the base point $s\in S(\mathbb C)$ once and for all.  

Moreover, if we choose $h_1,h_2 \in \pi_1(S(\bC),s)$ we can consider the action of the product $h_1h_2$. Recalling \Cref{columnAction}, we have
\begin{align*}
	\log_\sigma \xrightarrow{h_2} \log_\sigma + (u_1^{h_2},\ldots, u_{2g}^{h_2})\left(\begin{matrix} \omega_1\\ \vdots \\ \omega_{2g}\end{matrix}\right) \xrightarrow{h_1} \log_\sigma + (u_1^{h_1},\ldots, u_{2g}^{h_1})\left(\begin{matrix} \omega_1\\ \vdots \\ \omega_{2g}\end{matrix}\right) + (u_1^{h_2},\ldots, u_{2g}^{h_2})\rho(h_1)^\top\left(\begin{matrix} \omega_1\\ \vdots \\ \omega_{2g}\end{matrix}\right).
\end{align*}
Therefore, we obtain the cocycle relation
\begin{equation}\label{eq:cocycle}
    (u_1^{h_1h_2},\ldots, u_{2g}^{h_1h_2}) = (u_1^{h_1},\ldots, u_{2g}^{h_1}) + (u_1^{h_2},\ldots, u_{2g}^{h_2})\rho(h_1)^\top.
\end{equation}

Now, let us look at the simultaneous monodromy action of $\pi_1(S(\bC),s)$  on periods and logarithm. Thanks to \Cref{eq:cocycle}, using the functions described in \Cref{monPer} and \Cref{monLog}, we can provide a new representation
\[
\theta_\sigma: \pi_1(S(\bC),s) \rightarrow \textrm{SL}_{2g+1}(\bZ),
\]
where every matrix $\theta_\sigma(h)$ is of the form
\begin{equation}\label{simultaneousRepresentation}
	\theta_\sigma(h)=\left( \begin{matrix} \rho(h) & c(h) \\ 0 & 1 \end{matrix} \right),
\end{equation}
for $\rho(h) \in \textrm{Sp}_{2g}(\bZ), c(h) \in \bZ^{2g}$. Note that the matrix $\rho(h)$ acts on the periods and does not depend on $\sigma$. Moreover, the vector $c(h)$ encodes the action of $h$ on determinations of the logarithm $\log_\sigma$. 

\begin{defn}
The monodromy group of the section $\sigma$ is $M_\sigma:=\theta_\sigma(\pi_1(S(\bC),s))$.
\end{defn}

Let's continue considering an abelian scheme $\cA \rightarrow S$ and a non-zero section $\sigma:S \rightarrow \cA$. Generally neither a basis of periods nor a logarithm $\log_\sigma$ can be globally defined on the whole of $S(\bC)$ (see \Cref{wellDefPeriods} and \Cref{wellDefLog} below), but obviously they can be globally defined on the universal cover of $S(\bC)$. Studying the related monodromy problems corresponds to finding out the minimal topological cover of $S(\bC)$ on which both a basis of the periods and a logarithm of the section can be defined. With this in mind, we first call $S^* \rightarrow S(\bC)$ the minimal topological cover of $S(\bC)$ on which a basis for the period lattice can be globally defined and we set $S_\sigma \rightarrow S^*$ to be the minimal cover of $S^*$ where one can define the logarithm of $\sigma$. The tower of covers is represented in the diagram:
\begin{equation}\label{monodromyDiagram}
	S_\sigma \rightarrow S^* \rightarrow S(\bC).
\end{equation}
In particular, the group $\textnormal{Mon}(\cA)$ corresponds to the Galois group of the covering map $S^* \rightarrow S(\bC)$, while the group $M_\sigma$ corresponds to the Galois group of the covering map $S_\sigma\rightarrow S(\bC)$. Our interest is in studying the relative monodromy of the logarithm of a section with respect to the monodromy of periods, i.e. studying the covering map $S_\sigma \rightarrow S^*$. Topologically, this is the same as looking at the variation of logarithm via analytic continuation along loops on $S(\bC)$ which leave periods unchanged. With the notations of \Cref{monPer} and \Cref{simultaneousRepresentation} this corresponds to studying group defined below:

\begin{defn}
$M_\sigma^{\textnormal{rel}}:=\theta_\sigma(\ker{\rho})$ is called the \emph{relative monodromy group of $\sigma$}.
\end{defn}
The group $M_\sigma^{\textnormal{rel}}$ gives information on the ``pure monodromy'' of the Betti map, and this is a strong knowledge for obtaining transcendence results which for example allow to count torsion points in issues with ``unlikely intersection'' flavour. Note that $M_\sigma^{\textnormal{rel}}$ is clearly a subgroup of $\bZ^{2g}$; it is useful for applications knowing when this subgroup is trivial and how large it is.

\begin{expl}\label{exaTorsSec}
	Let us illustrate what happens in a simple case, i.e. when the section is torsion. Thus, let $\sigma:S \rightarrow \cA$ be a torsion section. By the properties of Betti map, any logarithm of $\sigma$ is a rational constant combination of periods; i.e.
	$$
	\log_\sigma = q_1 \omega_1 + \cdots + q_{2g} \omega_{2g},
	$$
	where $q_1, \ldots, q_{2g} \in \bQ$. Therefore, a loop which leaves unchanged periods via analytic continuation, leaves also unchanged the logarithm of such a section. In other words, the cover $S_\sigma \rightarrow S^*$ is trivial in this case and then $M_\sigma^\textnormal{rel}\cong \{0\}$.
\end{expl}

\begin{rem}
	Note that the Zariski-closures of the discrete groups $\textnormal{Mon}(\cA)$ and $M_\sigma$ as far as the kernel of the projection $\textrm{ker}(\overline{M_\sigma} \to \overline{\textnormal{Mon}(\cA)})$ are the differential Galois groups of some Picard-Vessiot extensions of $\bC(S)$ obtained with $\omega_1, \ldots, \omega_{2g}, \log_\sigma$ and their derivatives (see \cite{CH} for further details about differential Galois theory and Picard-Vessiot extensions). In these terms, the differential Galois group $\textrm{ker}(\overline{M_\sigma} \to \overline{\textnormal{Mon}(\cA)})$ was just studied in \cite{André}, \cite{André2}, \cite{Ber2}. Anyway, those results give no information on the relative monodromy of the logarithm over $S^*$ because they involve the Zariski closures of $M_{\sigma}$ and $\textnormal{Mon}(\cA)$: in fact, in \cite{CZ} the authors pointed out that it may happen that the group $\textrm{ker}(\overline{M_\sigma} \to \overline{\textnormal{Mon}(\cA)})$ contains quite limited information on $M_{\sigma}^\textrm{rel}$ since the former may be larger than expected in comparison with the latter. Thus, in order to obtain information on $M_\sigma^\textnormal{rel}$, we have to introduce considerations of different nature with respect to Bertrand's and André's theorems.
\end{rem}

Similarly to period functions, once we have locally defined the logarithm of a section as in \Cref{log} we can think about its analytic continuation through the whole of $S(\bC)$. As a first example, let us focus on the zero-section $\sigma_0$ which associates to each $s \in S(\bC)$ the origin $O_s$ of the corresponding fiber $\cA_s$. A logarithm of $\sigma_0$ is given by the zero function
\[
\log_{\sigma_0}:S(\bC) \rightarrow \bC^g, \qquad \log_{\sigma_0}(s)=0 \textrm{ for each } s \in S(\bC).
\]
Thus, in this case we can find a globally defined logarithm on the whole of $S(\bC)$: in fact it has no monodromy. For algebraic sections, this is the only case in which such a global logarithm exists. For the sake of completeness, we briefly resume a proof of this fact.

\begin{rem}\label{wellDefLog}
	Let $\cA \to S$ be an abelian scheme. We use Lang-Néron theorem \cite[Thorem 1]{LanNér} to show that a non-zero algebraic section $\sigma:S \rightarrow \cA$ which is not contained in the fixed part cannot admit a globally defined logarithm on $S(\bC)$. In fact, assume by contradiction that $\log_\sigma$ exists globally on $S(\bC)$, then $\frac{\sigma}{n}:= \exp \circ \left(\frac{1}{n}\log_\sigma\right)$ for any $n \in \bZ \setminus \{0\}$ is a holomorphic section such that $n \cdot \frac{\sigma}{n}=\sigma$. This means that $\sigma$ is infinitely divisible in $\Sigma(S)$ contradicting Lang-Néron theorem. 
\end{rem}

Fix a determination of $\log_\sigma$ on an open set $U$ containing the base point $s$ and use the notation $G:=\pi_1(S(\bC),s)$. Thanks to \Cref{eq:cocycle}, the map $c: h \mapsto (u_1^h,\ldots, u_{2g}^h)$ \emph{is a cocycle for the action of $G$ on $\bZ^{2g}$ described in \Cref{actionLog}}. Moreover, in the next proposition we prove that a section admitting a globally defined logarithm is characterized by the fact that the map $h \mapsto (u_1^h,\ldots, u_{2g}^h)$ is a coboundary for the aforementioned action, i.e. there exists a fixed vector $(u_1,\ldots,u_{2g}) \in \bZ^{2g}$ such that $(u_1^h,\ldots, u_{2g}^h) = (u_1,\ldots,u_{2g})(\rho(h)^\top-1_{2g})$ for all $h \in G$. Thus, the just defined cocycle determines an element in the cohomology group $H^1(G,\bZ^{2g})$ that describes the obstruction for a(n analytic) section to have a globally defined logarithm.

\begin{prop}\label{coboundary}
	Let $\sigma:S\rightarrow \cA$ be a holomorphic section and $\log_\sigma$ a determination of its logarithm over an open set $U\subset S(\bC)$. The section admits a globally defined logarithm on $S(\bC)$ if and only if the associated cocycle $h \mapsto (u_1^h,\ldots, u_{2g}^h)$ is a coboundary.
\end{prop}

\begin{proof}
	Suppose that $\sigma$ admits a globally defined logarithm $\ell:S(\bC) \rightarrow \bC^g$. Then the two determinations $\ell$ and $\log_\sigma$ over $U$ differ by a period $\omega:=n_1\omega_1+\cdots+n_{2g}\omega_{2g}$, i.e. $\log_\sigma=\ell + n_1\omega_1 + \cdots + n_{2g}\omega_{2g}$. For $h \in G$, we have
	\begin{align*}
		h\cdot\log_\sigma &= h\cdot (\ell + n_1\omega_1 + \cdots + n_{2g}\omega_{2g}) = \ell + (n_1,\ldots,n_{2g})\rho(h)^\top\left(\begin{matrix} \omega_1 \\ \vdots \\ \omega_{2g}\end{matrix}\right)=\\
		&= \log_\sigma + \left[(n_1,\ldots,n_{2g})\rho(h)^\top-(n_1,\ldots,n_{2g})\right] \left(\begin{matrix} \omega_1 \\ \vdots \\ \omega_{2g}\end{matrix}\right).
	\end{align*}
	Thus the corresponding cocycle is given by
	\[
    h\mapsto (n_1,\ldots,n_{2g})\rho(h)^\top-(n_1,\ldots,n_{2g})
    \]
	for $h \in G$ and a fixed pair $(n_1,\ldots,n_{2g}) \in \bZ^{2g}$, hence it is a coboundary.
	
	Viceversa, let us suppose to have a logarithm $\log_\sigma$ over $U$ and that there exists a fixed $2g$-tuple $(n_1,\ldots,n_{2g}) \in \bZ^{2g}$ such that
	\[
    h\cdot\log_\sigma = \log_\sigma + \left[(n_1,\ldots,n_{2g})\rho(h)^\top-(n_1,\ldots,n_{2g})\right] \left(\begin{matrix} \omega_1 \\ \vdots \\ \omega_{2g}\end{matrix}\right)
    \]
	for each $h \in G$. Let us define the function
	\[
    \ell:=\log_\sigma - n_1\omega_1 -\cdots- n_{2g} \omega_{2g},
    \]
	which is another determination of the logarithm of $\sigma$. Looking at the action of $G$ we obtain
	\begin{align*}
		h\cdot \ell &= h\cdot(\log_\sigma - n_1\omega_1 - \cdots - n_{2g} \omega_{2g}) =\\
		&=\log_\sigma + \left[(n_1,\ldots,n_{2g})\rho(h)^\top-(n_1,\ldots,n_{2g})\right] \left(\begin{matrix} \omega_1 \\ \vdots \\ \omega_{2g}\end{matrix}\right) - (n_1,\ldots,n_{2g})\rho(h)^\top \left(\begin{matrix} \omega_1 \\ \vdots \\ \omega_{2g}\end{matrix}\right)=\\
		&= \log_\sigma - n_1\omega_1 - \cdots - n_{2g} \omega_{2g} =\ell.
	\end{align*}
	Therefore, $\ell$ is a globally defined logarithm of $\sigma$ on $S(\bC)$.\\
\end{proof}

\section{Proof of the main results}\label{section3}

\subsection{Invariance results}\label{section3.1}

In this section we want to provide some preliminary results about the relative monodromy group $M_\sigma^\textnormal{rel}$ in order to prove that our main theorem are insensitive to some natural operations. First of all, let us notice that \Cref{mainThm1} is invariant under isogenies: this was proven in \cite[Theorem 3.6]{T1}. In what follows we describe other invariance results.

\begin{lem}\label{baseChange}
	\Cref{mainThm1} is invariant by finite base change: in other words, if $\varphi: \widetilde{S} \rightarrow S$ is a finite morphism and $\widetilde{\cA} \to \widetilde{S}$ is the pullback of $\cA\rightarrow S$ via $\varphi$, then \Cref{mainThm1}  for $\widetilde{\cA} \rightarrow \widetilde{S}$ implies \Cref{mainThm1}  for $\cA \rightarrow S$.
\end{lem}

\begin{proof}
	Let us denote by $\omega_1, \ldots, \omega_{2g}$ a basis of periods of $\cA \rightarrow S$. We can define $\widetilde{\omega_1}, \ldots, \widetilde{\omega_{2g}}$, a basis of periods of $\widetilde{\cA} \rightarrow \widetilde{S}$, by the equations
	\begin{equation}\label{eqn3}
		\widetilde{\omega_1}=\omega_1 \circ \varphi,\; \ldots,\; \widetilde{\omega_{2g}}=\omega_{2g} \circ \varphi.
	\end{equation}
	Fix two base points $\widetilde{s} \in \widetilde{S}(\bC)$ and $s \in S(\bC)$ such that $\varphi(\widetilde{s})=s$. We have the associated monodromy representations:
	\[
    \rho:\pi_1(S(\bC),s) \rightarrow \textrm{Mon}(\cA), \qquad \widetilde{\rho}:\pi_1(\widetilde{S}(\bC),\widetilde{s}) \rightarrow \textrm{Mon}(\widetilde{A}).
    \]
	By \Cref{eqn3}, we have $\widetilde{\rho}(h)=\rho(\varphi_*(h))$ for each $h \in \pi_1(\widetilde{S}(\bC),\widetilde{s})$ where $\varphi_*:\pi_1(\widetilde{S}(\bC),\widetilde{s})\rightarrow \pi_1(S(\bC),s)$ denotes the induced homomorphism between fundamental groups. In particular, we obtain
	\begin{equation}\label{eqn3_bis}
		\varphi_*(\ker{\widetilde{\rho}})=\ker{\rho} \cap \varphi_*(\pi_1(\widetilde{S}(\bC),\widetilde{s})).
	\end{equation}
	Let $\sigma:S \rightarrow \cA$ be a non-torsion section of $\cA \rightarrow S$. Since $\widetilde{\cA}\rightarrow \widetilde{S}$ is obtained as pullback of $\cA \rightarrow S$, then the abelian varieties $\widetilde{\cA}_{\widetilde{s}}$ and $\cA_s$ are canonically identified; therefore the pullback $\varphi^*\sigma:\widetilde{S} \rightarrow \widetilde{\cA}$ is a non-torsion section of $\widetilde{\cA} \rightarrow \widetilde{S}$. We have the associated monodromy representations
	\[
    \theta_\sigma:\pi_1(S(\bC),s) \rightarrow \textrm{SL}_{2g+1}(\bZ), \qquad \theta_{\varphi^*\sigma}:\pi_1(\widetilde{S}(\bC),\widetilde{s}) \rightarrow \textrm{SL}_{2g+1}(\bZ).
    \]
		
	Using again the fact that the abelian varieties $\widetilde{\cA}_{\widetilde{s}}$ and $\cA_s$ are canonically identified, we obtain that a determination of the logarithm of $\varphi^*\sigma$ at $\widetilde{s}$ can be defined by the equation
	\begin{equation}\label{eqn4}
		\log_{\varphi^*\sigma}(\widetilde{s}):=\log_\sigma(s).
	\end{equation}
	Hence, we have $\theta_{\varphi^*\sigma}(h)=\theta_\sigma(\varphi_*(h))$ for each $h \in \pi_1(\widetilde{S}(\bC),\widetilde{s})$. Now, let us consider the relative monodromy group
	\[
	M_{\varphi^*\sigma}^{\textnormal{rel}}=\theta_{\varphi^*\sigma}(\ker{\widetilde{\rho}}).
    \]
    By \Cref{eqn3_bis} and \Cref{eqn4}, we obtain
	\[
	M_{\varphi^*\sigma}^{\textnormal{rel}} = \theta_\sigma(\varphi_*(\ker{\widetilde{\rho}})) = \theta_{\sigma}(\ker{\rho} \cap \varphi_*(\pi_1(\widetilde{S}(\bC),\widetilde{s}))).
    \]
    Since $\ker{\rho} \cap \varphi_*(\pi_1(\widetilde{S}(\bC),\widetilde{s})) \subset \ker{\rho}$, we get $M_{\varphi^*\sigma}^{\textnormal{rel}} \subseteq M_{\sigma}^{\textnormal{rel}}$. The conclusion follows.\\
\end{proof}

Let us consider a non-torsion section $\sigma:S \rightarrow \cA$. Note that the representation $\theta_\sigma$ is defined in terms of \Cref{actionLog}, thus it depends on the branch of logarithm we fix. Here, we want to prove that $M_\sigma^\textnormal{rel}$ is independent of the choice of branch of $\log_\sigma$ and remains unchanged under some operations on the section.

\begin{prop}\label{relMonGrp}
	Let $\sigma: S \rightarrow \cA$ be a non-torsion section. Then
	\begin{itemize}
		\item[(i)] the group $M_\sigma^\textnormal{rel}$ does not depend on the choice of branch of $\log_\sigma$;
		
		\item[(ii)] the groups $M_\sigma^\textnormal{rel}$ and $M_{n\sigma}^\textnormal{rel}$ are isomorphic.
	\end{itemize}
\end{prop}

\begin{proof}
	$(i)$: Choose two branches of $\log_\sigma$, say $\ell_\sigma^1$ and $\ell_\sigma^2$ and denote by $M_{\sigma,1}^\textnormal{rel}$ and $M_{\sigma,2}^\textnormal{rel}$ the corresponding relative monodromy groups. Since the two branches $\ell_\sigma^1$ and $\ell_\sigma^2$ have to differ by a period, the conclusion follows: in fact, since any loop $\alpha$ in $S(\bC)$ whose homotopy class lies in $\ker{\rho}$ leaves periods unchanged, then $\ell_\sigma^1$ and $\ell_\sigma^2$ have the same variation by analytic continuation along $\alpha$; thus we get $M_{\sigma,1}^\textnormal{rel}=M_{\sigma,2}^\textnormal{rel}$.
	
	$(ii)$: Fixed a branch $\ell_\sigma$ of $\log_\sigma$, a determination $\ell_{n\sigma}$ of the logarithm of $n\sigma$ can be defined by the equation $ \ell_{n\sigma}=n\ell_\sigma$. For any loop $\alpha$ representing a homotopy class $h \in \pi_1(S(\bC),s)$ we have the corresponding variations:
	\[
	h\cdot \ell_{n\sigma}=\ell_{n\sigma}+\omega_\alpha^{n\sigma}, \qquad  h\cdot \ell_\sigma = \ell_\sigma+\omega_\alpha^\sigma,
	\]
	where $\omega_\alpha^{n\sigma}=n\omega_\alpha^\sigma$. Thus $M_\sigma^\textnormal{rel}$ and $M_{n\sigma}^\textnormal{rel}$ are isomorphic.\\
\end{proof}

\subsection{Non triviality of the relative monodromy}\label{section:3.2}

In this section we prove \Cref{mainThm1}, i.e. that $M_{\sigma}^{\textnormal{rel}}$ is non-trivial for all sections which ramify outside a fixed proper closed subset of the base. Recall that by  \Cref{baseChange}, we can assume that the abelian scheme $\phi:\cA \to S$ is endowed with a dominant morphism $p^\nu:S \to T^\nu$ coming from the normalization of $T$.

By the hypothesis of the theorem, the modular map $p \colon S \to T$ is assumed to be generically finite. Using \cite[Lemma 10.5]{Hart}, this implies that the locus $U_p \subseteq S$ where $p$ is finite and flat forms a dense Zariski-open subset. The non-triviality of $M_{\sigma}^{\textnormal{rel}}$ is preserved under restriction to dense Zariski-open subsets of the base, as the induced homomorphism between fundamental groups is surjective by \cite[Proposition 2.10]{kol}.  Therefore, after possibly replacing $S$ with $U_p$, we may assume that the morphism $p = \nu \circ p^\nu \colon S \to T$ is finite and flat. Observe that $p^\nu$ is quasi-finite by construction. Moreover, since $\nu$ is separated, the morphism $p^\nu$ is proper and consequently finite and surjective.
\begin{rem}\label{rem:case_ell}
    The restriction to $U_p$ is necessary because our argument relies on the ``trace operator'' for sections, as defined in \Cref{the_trace}. The key properties of this operator, particularly the one established in \Cref{propTrPu}, are valid when $p$ is flat and finite.
\end{rem}

Given a section $\sigma$ of $\cA/S$, the proof proceeds by assuming that the group $M_{\sigma}^{\textnormal{rel}}$ is trivial and showing that $\sigma$ must then be unramified. The key idea involves constructing an ``unramified family'', i.e. an unramified base change of the universal family, associated with the abelian scheme, which ensures that the cohomology class arising from the logarithm of $\sigma$ ``descends''. Applying the Lefschetz hyperplane theorem then imposes strong restrictions on the logarithm of $\sigma$. Combining this with \Cref{prop:unramifiedMiddle} and \Cref{cor:divisible_ram}, we conclude that $\sigma$ is indeed unramified.

In order to construct the aforementioned unramified family, we consider the  map $p^\nu:S \to T^\nu$ and we look at the induced homomorphism on fundamental groups $p^\nu_*:\pi_1(S(\bC),s) \to \pi_1(T^\nu(\bC),x)$ where $s \in S(\bC)$ and $x \in T^\nu(\bC)$ are two base points such that $x=p^\nu(s)$. By \cite[Corollary 12.19]{Lee2} the inclusion $p^\nu_*\pi_1(S(\bC),s) \subseteq \pi_1(T^\nu(\bC),x)$ gives rise to a finite topological cover $q': S'(\bC) \to T^\nu(\bC)$ that  satisfies
\begin{equation}\label{liftProp}
q'_*\pi_1(S'(\bC),s')=p^\nu_*\pi_1(S(\bC),s),
\end{equation}
where $s'$ is a fixed base point with $q'(s')=x$. By the Riemann existence theorem  (see \cite[Exposé XII, Théorème~5.1]{SGA1}) at the level of schemes $q':S'\to T^\nu$ is a finite \'etale cover. By the lifting theorem for \'etale fundamental groups and the relation with the topological fundamental group we can lift $p^\nu$ to a morphism $q:S(\bC)\rightarrow S'(\bC)$ such that $q(s)=s'$. Considering the corresponding abelian schemes we obtain the diagram
\begin{equation}\label{factorScheme}
	\begin{tikzcd}
		\cA(\bC) \arrow{d} \arrow{r} & \cA'(\bC) \arrow{d} \arrow{r} & \fA_{g,T^\nu}(\bC) \arrow{d}\\
		S(\bC) \arrow{r}{q} \arrow[bend right, swap]{rr}{p^\nu}& S'(\bC) \arrow{r}{q'} & T^\nu(\bC).
	\end{tikzcd}
\end{equation}
Notice that the abelian scheme $\cA' \to S'$ is an unramified family by construction, since $q':S'(\bC) \to T^\nu(\bC) \subseteq \bA_g(\bC)$ is in particular an unramified morphism.

\begin{rem}\label{algTopPropConstr}
    We want to point out some algebraic and topological properties of the above construction. By \cite[Exposé I, Corollaire~10.2]{SGA1} the finite étale morphism $q'$ is obtained as normalization of $T^\nu$ in the field extension $\bC(S')$ of $\bC(T^\nu)$. By universal property of normalization in a field extension, the morphism $q$ is dominant. Since $p$ and $q'$ are both finite morphisms then $q$ is quasi-finite and proper, hence a finite surjective morphism as well. Additionally, since $S'$ is connected and $T^\nu$ is integral and normal it follows that $S'$ is normal and irreducible (see \cite[Tag 0BQL]{stacks-project}).

    From the topological point of view, observe that the induced homomorphism $q_*:\pi_1(S(\bC),s) \to \pi_1(S'(\bC),s')$ is surjective. This follows from the diagram
    \begin{equation}
    \begin{tikzcd}
	   \pi_1(S(\bC),s) \arrow{r}{q_*} \arrow[bend right, swap]{rr}{p^\nu_*}& \pi_1(S'(\bC),s') \arrow{r}{q'_*} & p^\nu_*\pi_1(S(\bC),s),
    \end{tikzcd}
    \end{equation}
    where $q'_*:\pi_1(S'(\bC),s') \to p^\nu_*\pi_1(S(\bC),s)$ is an isomorphism.
\end{rem}

Let's fix some notations. We denote by $\exp$ and $\exp'$ the abelian exponential maps of the schemes $\cA \to S$ and $\cA' \to S'$, respectively. Moreover, recall that $\Sigma^{\an}$ is the sheaf of holomorphic sections of the scheme $\cA \to S$, while $\Sigma(S)$ denotes its algebraic sections. Analogously, we denote by $\Sigma'^{,\an}$ and $\Sigma'(S')$ the sheaf of holomorphic sections and the algebraic sections of the abelian scheme $\cA' \to S'$, respectively. Now, we want to define a way of mapping sections of $\cA \to S$ to sections of $\cA'\to S'$, and viceversa.

First of all we define the the pull-back operator: in fact notice that for any $\tau' \in \Sigma'^{,\an}(S')$ the composition $\tau' \circ q$ defines a holomorphic map $S(\bC) \to \cA'(\bC)$; after canonically identifying the fibers of the type $\cA_y$ and $\cA'_{q(y)}$, we obtain an element $q^*(\tau') \in \Sigma^{\an}(S)$. In other words, we obtain a morphism of abelian groups
\[
q^*: \Sigma'^{,\an}(S') \to \Sigma^{\an}(S).
\]
In addition if $\tau'$ is algebraic then $q^*(\tau')$ is still algebraic, i.e. $q^*(\Sigma'(S'))\subseteq \Sigma(S)$.

In the opposite direction, since $q$ is a finite flat morphism the notion of degree is well-behaved with respect to the ramification index (see \cite[Proposition 12.21]{GorWed}). Thus, we can define the trace operator
\[
\Tr: \Sigma^{\an}(S) \to \Sigma'^{,\an}(S')
\]
in the following way: we already noticed that for each $y \in S(\bC)$ the fibers $\cA_y$ and $\cA'_{q(y)}$ are canonically identified, so for $\tau \in \Sigma^{\an}(S)$ we define $\Tr(\tau):S'(\bC) \to \cA'(\bC)$ by setting
\begin{equation}\label{the_trace}
\Tr(\tau)(y'):=\sum_{y \in q^{-1}(y')} m_y\tau(y) \in \cA'_{y'},
\end{equation}
where $m_y$ is the ramification index of $q$ at $y$. We have $\Tr(\Sigma(S)) \subseteq \Sigma'(S')$. Moreover, we point out that the following property holds for each holomorphic section $\tau':S'(\bC) \rightarrow \cA'(\bC)$:
\begin{equation}\label{propTrPu}
    \Tr \circ \, q^*(\tau') = \deg{q}\cdot \tau'.
\end{equation}

It is clear that the trace operator and the pullback operator can be defined also for the sections of the vector bundles $\left(\cO^{\an}_S\right)^{\oplus g}$ and $\left(\cO^{\an}_{S'}\right)^{\oplus g}$, i.e. for vectors of holomorphic functions on $S$ and $S'$. They commute with the exponentials, in the sense that for all holomorphic functions $f:S(\bC) \rightarrow \bC^g$ and $f':S'(\bC) \rightarrow \bC^g$ we get
\begin{equation}\label{propTrExp}
    \Tr(\exp(f)) = \exp'(\Tr(f)) \qquad \textrm{and} \qquad \exp(q^*(f')) = q^*(\exp'(f')).
\end{equation}

Now we want to consider some cohomology groups related to the monodromy of logarithms. In our situation the bases of abelian schemes are too large to be able of having control on their cohomology. If we consider an algebraic curve $C \subset S(\mathbb C)$, we can think of restricting the diagram of \Cref{factorScheme} to $C$: more precisely, if we denote $C':=q(C)$ and $D:=q'(C')$ we get a new diagram:
\begin{equation}\label{factorSchemeRestr}
\begin{tikzcd}
	\cA_{|_C}(\bC) \arrow{d} \arrow{r} & \cA'_{|_{C'}}(\bC) \arrow{d} \arrow{r} & \fA_{g|_D}(\bC) \arrow{d}\\
	C \arrow{r}{q_C} \arrow[bend right, swap]{rr}{p_C}& C' \arrow{r}{q'_{C'}} & D,
\end{tikzcd}
\end{equation}
where $q_C, q'_{C'}$ and $p_C$ denote the restrictions of the corresponding maps. In this situation, if $C'$ and $D$ are still of dimension $1$, we can clearly consider the analogous notions of trace operator and pullback operator as well as the restriction of a section $\sigma:S \to \cA$ to $C$, which will be denoted by $\sigma_C$. We also have analogous notions of monodromy representations and monodromy groups, e.g. we will write $\rho_C, \theta_{\sigma_C}, \textrm{Mon}(\cA|_C)$ and $M^\textrm{rel}_{\sigma_C}$ for the obvious objects associated to the scheme $\cA|_C \to C$ and to the section $\sigma_C:C\to \cA|_C$; we do the same for $C'$. Moreover, if we denote by $i: C \hookrightarrow S(\bC)$ and $i': C' \hookrightarrow S'(\bC)$ the natural embeddings and choose base points $s \in C$ and $s' \in C'$, from a topological point of view we obtain the following diagram involving the induced homomorphisms between fundamental groups (for the surjectivity of the upper horizontal arrow see \Cref{algTopPropConstr}):
\begin{equation}\label{topDiagram}
\begin{tikzcd}
	\pi_1(S(\bC),s) \arrow[two heads]{r}{q_*} & \pi_1(S'(\bC),s')\\
	\pi_1(C,s) \arrow{r}{{q_{C}}_\ast} \arrow{u}{i_*} & \pi_1(C',s') \arrow[swap]{u}{i'_*}.
\end{tikzcd}
\end{equation}
Anyway, this kind of restriction causes the loss of much information from the starting schemes. With the next result we want to overcome the problem of having large bases, by showing that we can restrict the abelian schemes to some suitable curves lying into the bases so that the topological and monodromy properties are preserved.

\begin{lem}\label{LefCurve}
    There exists a curve $C \subseteq S(\bC)$ with the following properties:
    \begin{enumerate}
        \item[$(i)$] the morphism $q_C:C \to C'$ appearing in \Cref{factorSchemeRestr} is finite, the curves $C$ and $C'$ are affine and the abelian schemes $\cA_{|_C} \to C$, $\cA'_{|_{C'}} \to C'$ are without fixed part; 

        \item[$(ii)$] $\Tr(\sigma_C)=\Tr(\sigma)|_{C'}$ for any $\sigma \in \Sigma^\textnormal{an}(S)$ and $\textnormal{Tr}\circ q_C^*(\tau) = \deg{q} \cdot \tau$ for any $\tau \in \Sigma^\textnormal{an}(C')$;

        \item[$(iii)$] the homomorphisms $i_*$ and $i'_*$ appearing in \Cref{topDiagram} are surjective;

        \item[$(iv)$] if we denote by $\rho'$ the monodromy representation of the abelian scheme $\cA'\to S'$, then we have $\rho_C=\rho\circ i_*, \rho_{C'}=\rho'\circ i'_*$ and $\textrm{Mon}(\cA)=\textrm{Mon}(\cA|_C), \textrm{Mon}(\cA')=\textrm{Mon}(\cA'|_{C'})$;

        \item[$(v)$] $\ker{{q_{C}}_\ast}\subseteq \ker{\rho_C}$;

        \item[$(vi)$] $\theta_{\sigma_C}=\theta_\sigma \circ i_*$ and $\theta_\sigma(\pi_1(S(\bC),s)) = \theta_{\sigma_C}(\pi_1(C,s))$, for any $\sigma \in \Sigma^\textnormal{an}(S)$;

        \item[$(vii)$] if $M_\sigma^\textrm{rel}=\{0\}$, then $M_{\sigma_C}^\textrm{rel}=\{0\}$, for any $\sigma \in \Sigma^\textnormal{an}(S)$.
    \end{enumerate}
\end{lem}

\begin{proof}
    \begin{enumerate}
        \item[$(i)$] In order to construct the curve $C$ let us consider the general hyperplane section $Y=H \cap S(\bC)$ and choose a base point $s \in Y$. By Lefschetz theorem on quasi-projective varieties \cite[Part II, 1.2, Theorem]{GorMac}, the natural map $\pi_1(Y,s) \to \pi_1(S(\bC),s)$ between fundamental groups is surjective. After iterating the process of taking such generic hyperplane sections we obtain an irreducible curve $C_1$ with the property that the natural map $\pi_1(C_1,s) \to \pi_1(S(\bC),s)$ is surjective for a base point $s \in C_1$. Note that $C_1$ can be assumed to be affine up to remove a finite number of points and using \cite[Proposition 2.10]{kol}. The restriction of $q$ to $C_1$ defines a morphism $q_{C_1}$ which is finite: in fact, since every closed embedding is a finite morphism and $q:S \to S'$ is a finite morphism (see \Cref{algTopPropConstr}), then the restriction $q_{C_1}$ is finite. We define $C':=q_{C_1}(C_1)$ and $C:=q^{-1}(C')$. Notice that $C'$ is an affine curve since $q_{C_1}$ is finite and surjective; moreover, $C$ could be no longer irreducible but it is still affine because $q$ is an affine morphism (since it is finite). Furthermore, since $q:S \to S'$ is finite then it is closed. We have a closed embedding $i:C \hookrightarrow S(\bC)$, and we can conclude that the restriction $q_C:C \to C'$ is a finite morphism. The abelian scheme $\cA_{|_C} \to C$ is without fixed part by genericity of the hyperplane sections; this implies that also $\cA'_{|_{C'}} \to C'$ is without fixed part.
        


        \item[$(ii)$] This properties are true by construction, since $C=q^{-1}(C')$.

        \item[$(iii)$] By $(i)$ we have a surjective homomorphism $\pi_1(C_1,s) \to \pi_1(S(\bC),s)$. Recall that $C_1$ is an irreducible component of $C$. Since the inclusion map $C_1 \hookrightarrow S(\bC)$ factors through $C$, then the homomorphism $\pi_1(C_1,s) \to \pi_1(S(\bC),s)$ factors through $\pi_1(C,s)$. Thus, $i_*:\pi_1(C,s) \to \pi_1(S(\bC),s)$ is surjective. Then, the surjectivity of $i'_*:\pi_1(C',s') \to \pi_1(S'(\bC),s')$ follows by \Cref{topDiagram}.

        \item[$(iv)$] Just recall that for abelian schemes obtained as base change we define periods as in \Cref{periodFunctions}, so that we have $\rho_C=\rho\circ i_*$ and $\rho_{C'}=\rho'\circ i'_*$. Therefore, the claim follows from $(iii)$.

        \item[$(v)$] Let $h \in \ker{q_{C*}}$. Since we can define periods of $\cA|_C \to C$ as in \Cref{periodFunctions} and $q_{C*}(h)$ has a trivial monodromy action on the periods of $\cA'|_{C'} \to C'$, then we get $h \in \ker{\rho_C}$; hence the claim.

        \item[$(vi)$] Since we have the embedding $i:C \hookrightarrow S(\bC)$, we clearly have $\theta_{\sigma_C}(\pi_1(C),s) \subseteq \theta_\sigma(\pi_1(S(\bC),s))$. In the other direction, by $(iv)$ it is enough to show that each variation of a logarithm of $\sigma$ induced by an element of $\pi_1(S(\bC),s)$ can be obtained as a variation of a logarithm of $\sigma_C$ induced by the monodromy action of some element of $\pi_1(C,s)$; and this is true thanks to the surjectivity of $i_*$ which was proven in part $(iii)$. This proves the claim.

        \item[$(vii)$] Let us assume $M_\sigma^\textrm{rel}=\{0\}$ and consider $h \in \ker{\rho_C}$. By part $(iv)$ we have $\rho_C=\rho\circ i_*$, and this implies that $i_*(h) \in \ker{\rho}$. The assumption $M_\sigma^\textrm{rel}=\{0\}$ implies that $\theta_\sigma(i_*(h))=0$. By part $(vi)$ we have $\theta_{\sigma_C}=\theta_\sigma \circ i_*$ which implies $\theta_{\sigma_C}(h)=\theta_\sigma(i_*(h))=0$, i.e. $M_{\sigma_C}^\textrm{rel}=\{0\}$.
    \end{enumerate}
\end{proof}

Now, let us consider an algebraic curve $C \subset S(\mathbb C)$ satisfying \Cref{LefCurve}. We want to consider the restriction of the exact sequence in \Cref{exSeqSch} to the curve $C$. Since $C$ is a Stein space, by \cite[Theorem 5.3.1]{For} the restriction $\cLie(\cA)|_C$ is isomorphic to $\left(\mathcal O^{\an}_C\right)^{\oplus g}$, where $\mathcal O^{\an}_C$ denotes the sheaf of holomorphic functions on $C$. Therefore, we obtain the exact sequence
\begin{equation}\label{exactSeqSheaves}
    0 \rightarrow \Lambda_C \rightarrow \left(\mathcal O^{\an}_C\right)^{\oplus g} \rightarrow \Sigma^{\an}|_C \rightarrow 0.
\end{equation}

Let us study the exact sequence in cohomology groups induced by \Cref{exactSeqSheaves}. By \Cref{wellDefPeriods} we can conclude that no non-zero period can be globally defined on $C$. In other words, we get $H^0(C,\Lambda_C)=0$. In addition, since $C$ is a Stein space we have
\[
H^1(C,\left(\mathcal O^{\an}_C)^{\oplus g}\right)=\bigoplus_{i=1}^g H^1(C,\left(\mathcal O^{\an}_C)\right) = 0.
\]
Thus, we obtain the exact sequence of cohomology groups
\begin{equation}\label{exactSeqCoh}
	0 \rightarrow H^0(C,(\mathcal O^{\an}_C)^{\oplus g}) \rightarrow H^0(C,\Sigma^{\an}|_C) \rightarrow H^1(C,\Lambda_C) \rightarrow 0.
\end{equation}

By \Cref{coboundary}, we know that the obstruction to define a global logarithm of a section is given by some Galois cohomology class. We want to explain the interplay between the relevant Galois cohomology group and the exact sequence in \Cref{exactSeqCoh}. Let's fix the notation $G:=\pi_1(C,s)$ for a base point $s \in C$ and denote by $\rho_C:G \to \textrm{Mon}(\cA|_C)\subseteq \textrm{Sp}_{2g}(\bZ)$ the monodromy representation associated to the abelian scheme $\cA|_C \to C$. We give a description of the group $H^1(C, \Lambda_C)$ in terms of the Galois cohomology with respect to the action of $G$ on $\bZ^{2g}$ induced by the projection $\rho_C$. We have the following

\begin{prop}\label{cohoIso}
	With the above notation the Čech cohomology group $H^1(C,\Lambda_C)$ is canonically isomorphic to the group $H^1(G,\bZ^{2g})$, where the action of $G$ on $\bZ^{2g}$ is induced by the projection $\rho_C: G \rightarrow \textrm{Mon}(\cA|_C)$.
\end{prop}

\begin{proof}
It follows from \cite[Proposition 2.1.1]{Ach} since the universal cover of $C$ is contractible. One can also generalize the independent proof given in \cite[Proposition 4.1]{CZ}. 
\end{proof}

Therefore, thanks to the exact sequence \Cref{exactSeqCoh} and \Cref{cohoIso} we obtain a surjective group homomorphism
\[
\Sigma^{\an}|_C(C) \ni \tau \mapsto [\tau] \in H^1(C,\Lambda_C) \simeq H^1(G,\bZ^{2g})
\]
associating to every holomorphic section a cohomology class, which measures the obstruction of having global abelian logarithm.

\begin{rem}\label{sameForC'}
    Notice that by \Cref{LefCurve} also the curve $C'$ is a Stein space. Moreover, by \Cref{wellDefPeriods} we still obtain  $H^0(C',\Lambda_{C'})=0$. Therefore, the analogous results of \Cref{exactSeqCoh} and \Cref{cohoIso} hold if we replace the curve $C$ by the curve $C'$. In particular, for some base point $s' \in C'$ we have the surjective group homomorphism
    \[
    \Sigma'^{,\an}|_{C'}(C') \ni \tau' \mapsto [\tau'] \in H^1(C',\Lambda_{C'}) \simeq H^1(\pi_1(C',s'),\bZ^{2g}),
    \]
    where the action of $\pi_1(C',s')$ on $\bZ^{2g}$ is induced by the monodromy representation $\rho_{C'}: \pi_1(C',s') \rightarrow \textrm{Mon}(\cA'|_{C'})$.
\end{rem}

\begin{rem}\label{injectiveCoho}
    Let us point out the relationship between the cohomology groups $H^1(\pi_1(S'(\bC),s'),\bZ^{2g})$ and $H^1(\pi_1(C',s'),\bZ^{2g})$. By part $(iii)$ of \Cref{LefCurve} the homomorphism $i'_*: \pi_1(C',s') \to \pi_1(S',s')$ is surjective and this induces a homomorphism
    \[
    f:H^1(\pi_1(S'(\bC),s'),\bZ^{2g}) \to H^1(\pi_1(C',s'),\bZ^{2g});
    \]
    more precisely, if $c': \pi_1(S'(\bC),s') \to \bZ^{2g}$ is a cocycle representing some cohomology class $[c']_{S'}$ in $H^1(\pi_1(S'(\bC),s'),\bZ^{2g})$ we can define $f([c']_{S'})$ as the cohomology class in $H^1(\pi_1(C',s'),\bZ^{2g})$ of the cocycle $c'':\pi_1(C',s') \to \bZ^{2g}$ defined as $c''(h):=c'(i'_*(h))$. The analogous relation holds for the cohomology groups $H^1(\pi_1(S(\bC),s),\bZ^{2g})$ and $H^1(\pi_1(C,s),\bZ^{2g})$.
\end{rem}

We are now ready to prove the first main result of the paper:

\paragraph{Proof of \Cref{mainThm1}.}
    Let $C$ be a curve satisfying \Cref{LefCurve}. We assume $M_\sigma^\textrm{rel}=\{0\}$ and we prove that $\sigma$ is unramified over $S$. By part $(vii)$ of \Cref{LefCurve} this implies $M_{\sigma_C}^\textrm{rel}=\{0\}$, or equivalently $\ker{\rho_C}\subseteq \ker{\theta_{\sigma_C}}$. Let ${q_{C}}_\ast:\pi_1(C,s) \to \pi_1(C',s')$ be the homomorphism appearing in \Cref{topDiagram}. By part $(v)$ of \Cref{LefCurve} we have $\ker{{q_{C}}_\ast}\subseteq \ker{\rho_C}$; thus, we have $\ker{{q_{C}}_\ast}\subseteq \ker{\theta_{\sigma_C}}$. Introducing the notation $\overline{G}:={q_{C}}_\ast\pi_1(C,s)$ we obtain that there exists a homomorphism $f:\overline{G}\rightarrow \textrm{SL}_{2g+1}(\bZ)$ such that the diagram
	\begin{equation}\label{factorDiag}
    \begin{tikzcd}
		\pi_1(C,s) \arrow[two heads]{r}{{q_{C}}_\ast} \arrow[swap]{rd}{\theta_{\sigma_C}} & \overline{G} \arrow[dashed]{d}{f}\\
		& \textrm{SL}_{2g+1}(\bZ)
	\end{tikzcd}
    \end{equation}
	commutes; in other words, $\theta_{\sigma_C}$ factors through ${q_{C}}_\ast$. By using the same notation as in \Cref{simultaneousRepresentation}, this means that for $h \in \pi_1(C,s)$ the vector $c(h)$ only depends on the element $\overline{h}:={q_{C}}_\ast(h) \in \overline{G}$. Then we obtain a cocycle
    \begin{equation}\label{inducedCocycle}
        \overline{G}\ni \overline{h} \mapsto c(h) \in \bZ^{2g},
    \end{equation}
	representing a certain cohomology class in $H^1(\overline{G},\bZ^{2g})$ which we denote with $[\sigma_C]_{\overline{G}}$. In other words, we have obtained an injective homomorphism $H^1(\pi_1(C,s),\bZ^{2g}) \hookrightarrow H^1(\overline{G},\bZ^{2g})$ which can be used to define a homomorphism
    \[
    \xi: H^0(\Sigma^{\an}|_C,C) \to H^1(\overline{G},\bZ^{2g}),
    \]
    which to a section $\tau$ over $C$ associates a cohomology class which we denote by $[\tau]_{\overline{G}}$. We want to show that $[\sigma_C]_{\overline{G}}$ can be canonically identified with a cohomology class $[\sigma]_{C'}$ in $H^1(\pi_1(C',s'),\bZ^{2g})$.

    To this end, let us construct some homomorphisms between cohomology groups. Since we are assuming $M_\sigma^\textrm{rel}=\{0\}$ and the map $q_*:\pi_1(S(\bC),s) \to \pi_1(S'(\bC),s')$ is surjective by \Cref{algTopPropConstr}, with the same reasoning as in \Cref{factorDiag} we obtain an injective homomorphism
    \[
    \zeta_1:H^1(\pi_1(S(\bC),s),\bZ^{2g}) \hookrightarrow H^1(\pi_1(S'(\bC),s'),\bZ^{2g})
    \]
    which to a cohomology class $[\tau]_S$ associates a cohomology class denoted with $[\tau]_{S'}$. By \Cref{injectiveCoho} we have a homomorphism
    \[
    \zeta_2:=f: H^1(\pi_1(S'(\bC),s'),\bZ^{2g}) \to H^1(\pi_1(C',s'),\bZ^{2g}).
    \]
    By composition we obtain a map
    \[
    \zeta_3:=\zeta_2\circ \zeta_1:H^1(\pi_1(S(\bC),s),\bZ^{2g}) \to H^1(\pi_1(C',s'),\bZ^{2g})
    \]
    which to a class $[\tau]_S$ associates a class denoted with $[\tau]_{C'}$. Since $\overline{G}$ is a subgroup of $\pi_1(C',s')$ the restriction homomorphism induces a homomorphism between cohomology groups
    \[
    \eta:H^1(\pi_1(C',s'),\bZ^{2g}) \to H^1(\overline{G},\bZ^{2g}).
    \]

    Notice that we cannot identify the cohomology groups $H^1(\overline{G},\bZ^{2g})$ and $H^1(\pi_1(C',s'),\bZ^{2g})$, but we want to prove that $\eta$ becomes an isomorphism when restricted to cohomology classes coming from global sections in $\Sigma(S)$. Thus, let us denote by $r:H^0(S,\Sigma^{\an}) \to H^0(C,\Sigma^{\an}|_C)$ the restriction map which to a section $\tau$ over $S$ associates the restriction $\tau_C$. Define the homomorphism
    \[
    \zeta:r(H^0(S,\Sigma^{\an})) \to H^1(\pi_1(C',s'),\bZ^{2g}), \qquad \tau_C \mapsto [\tau]_{C'},
    \]
    where $\tau \in \Sigma(S)$ is such that $r(\tau)=\tau_C$: notice that this map is well-defined by \Cref{LefCurve}. In fact, if $\tau_1, \tau_2 \in \Sigma(S)$ are such that $\tau_{1,C} = \tau_{2,C}$, then $\theta_{\tau_1}=\theta_{\tau_2}$ by part $(iii)$ and $(vi)$ of \Cref{LefCurve}. This implies $[\tau_1]_S=[\tau_2]_S$ and so $[\tau_1]_{C'}=[\tau_2]_{C'}$.

    Furthermore, by restriction of $\xi$ we define the homomorphism
    \[
    \xi':=\xi|_{r(H^0(S,\Sigma^{\an}))}: r(H^0(S,\Sigma^{\an})) \to H^1(\overline{G},\bZ^{2g}), \qquad \tau_C \mapsto [\tau_C]_{\overline{G}}.
    \]
    We claim that $\textrm{Im}(\zeta) \cong \textrm{Im}(\xi|_{r(H^0(S,\Sigma^{\an})})$. The restriction of $\eta$ to $\textrm{Im}(\zeta)$ induces a homomorphism
    \[
    \eta':=\eta|_{\textrm{Im}(\zeta)}: \textrm{Im}(\zeta) \to H^1(\overline{G},\bZ^{2g}), \qquad [\tau]_{C'} \mapsto [\tau_C]_{\overline{G}},
    \]
    where $r(\tau)=\tau_C$. Thanks to \Cref{topDiagram}, we have $\textrm{Im}(\eta') = \textrm{Im}(\xi')$ and the following commutative diagram:
    \[
    \begin{tikzcd}
        r(H^0(S,\Sigma^{\an})) \arrow{r}{\zeta} \arrow[swap]{dr}{\xi'} & \textrm{Im}(\zeta)\arrow{d}{\eta'}\arrow[r, phantom,"\subseteq"] & H^1(\pi_1(C',s'),\bZ^{2g}) \\
         & \textrm{Im}(\xi') \arrow[r, phantom,"\subseteq"] & H^1(\overline{G},\bZ^{2g}).
    \end{tikzcd}
    \]
    Let's prove that $\eta'$ is an isomorphism. Notice that it is well-defined and surjective by construction, so it's enough to prove that it is injective. To this end, consider $\tau_1, \tau_2 \in \Sigma(S)$ such that $[\tau_{1,C}]_{\overline{G}} = [\tau_{2,C}]_{\overline{G}}$. This implies
    \[
    [\tau_{1,C} - \tau_{2,C}]_{\overline{G}} = [\tau_{1,C}]_{\overline{G}} - [\tau_{2,C}]_{\overline{G}} = 0.
    \]
    Thanks to the injective homomorphism obtained by \Cref{inducedCocycle}, we get $[\tau_{1,C} - \tau_{2,C}]_{C}=0$. By part $(vi)$ of \Cref{LefCurve}, we also obtain $[\tau_1 - \tau_2]_S = 0$. By applying $\zeta_3$ we get $[\tau_1 - \tau_2]_{C'}=0$, which means that $\eta'$ is injective.
    
    Thus, by what we have just proven and by \Cref{sameForC'} the section $\sigma_C$ determines a cohomology class $[\sigma_C]_{C'}$ in $H^1(\pi_1(C',s'),\bZ^{2g})\cong H^1(C',\Lambda_{C'})$. Again thanks to \Cref{sameForC'}, we have that there exists a (possibly transcendental) section $\sigma':C' \rightarrow \cA'|_{C'}$ whose cohomology class $[\sigma']$ in $H^1(\pi_1(C',s'),\bZ^{2g})$ coincides with the class $[\sigma_C]_{C'}$ defined above, i.e. $[\sigma']=[\sigma_C]_{C'}$. Therefore, by \Cref{coboundary} we obtain a section of the form $\sigma_C - q_C^*(\sigma') \in \Sigma^{\an}(C)$ that admits a logarithm; let us write
	\begin{equation}\label{equationExp}
		\sigma_C - q_C^*(\sigma') =\exp(f),
	\end{equation}
	for some holomorphic function $f:C \rightarrow \bC^g$.
	
	Now, by using properties of the operators $q_C^*$ and $\Tr$, we will get a contradiction. Applying first the trace operator and then the pullback operator, by part $(ii)$ of \Cref{LefCurve} we obtain:
    \[
    q^*_C\Tr(\sigma_C) - \deg{q}\cdot q^*_C(\sigma')=\exp(q^*_C\Tr(f)).
    \]
    By using \Cref{equationExp}, we finally obtain
    \[
    q^*_C\Tr(\sigma_C)-\deg{q}\cdot \sigma_C = \exp(q^*_C\Tr(f)-\deg{q}\cdot f).
    \]
    Thus, the section $q^*_C\Tr(\sigma_C)-\deg{q}\cdot \sigma_C$ is algebraic and admits a global logarithm. Notice that
    \[
    q^*_C\Tr(\sigma_C)-\deg{q}\cdot \sigma_C = (q^*\Tr(\sigma)-\deg{q}\cdot \sigma)_{|_C}.
    \]
    By part $(vi)$ of \Cref{LefCurve}, the section $q^*\Tr(\sigma)-\deg{q}\cdot \sigma$ is also an algebraic section admitting a global logarithm on $S(\bC)$. By \Cref{wellDefLog} we must have $\deg{q}\cdot\sigma=q^*\Tr(\sigma)$. By \Cref{prop:unramifiedMiddle}, the section $\deg{q}\cdot\sigma$ is unramified. By \Cref{cor:divisible_ram} we conclude that $\sigma$ is an unramified section.\\
    \qed



\subsection{Rank of relative monodromy}\label{section:3.3}

Here, we prove \Cref{mainThm2}. 

We keep the same notations as above, in particular recall that $\rho: \pi_1(S(\bC),s) \to \textrm{Mon}(\cA)$ denotes the monodromy representation associated to the abelian scheme $\phi: \cA \to S$, which we are assuming to induce an irreducible action of $\textrm{Mon}(\cA)$ on the lattice of periods. In what follows we identify the $\bZ$-module $M_\sigma^\textrm{rel}$ with the corresponding $\bZ$-submodule of the lattice of periods.

\paragraph{Proof of \Cref{mainThm2}}
    Let us proceed by contradiction by assuming that $M_\sigma^\textrm{rel}$ is a submodule of $\bZ^{2g}$ of rank $g'$ with $g' < 2g$. By the non-triviality of $M_\sigma^\textrm{rel}$ proved in \Cref{mainThm1} we have $g'>0$. In other words, this means that for every $h \in H:=\ker{\rho}$ the logarithm $\log_\sigma$ is transformed by $h$ as
    \begin{equation*}
        \log_\sigma \xmapsto{h} \log_\sigma + u_1^h\mu_1 + \cdots + u_{g'}^h\mu_{g'},
    \end{equation*}
    where $\mu_1, \ldots, \mu_{g'}$ are fixed non-zero periods (depending on $\sigma$). Recall that, for $k \in \pi_1(S(\bC),s)$ the logarithm $\log_\sigma$ will be sent by $k$ to a new determination of the form
    \[
    \log_\sigma + v_1\omega_1 + \cdots + v_{2g} \omega_{2g},
    \]
    where $v_1, \ldots, v_{2g} \in \bZ$ and $\omega_1, \ldots, \omega_{2g}$ denote period functions. Since we are assuming that the action of $\textrm{Mon}(\cA)$ is irreducible, there exists $k \in \pi_1(S(\bC),s)$ such that $\rho(k) \cdot M_\sigma^{\textrm{rel}} \nsubseteq M_\sigma^{\textrm{rel}}$, i.e. the action of $\rho(k)$ does not preserve the relative monodromy group. Therefore, we can fix $k \in \pi_1(S(\bC),s)$ and $h \in H$ such that $\rho(k^{-1})\cdot (u_1^h\mu_1 + \cdots + u_{g'}^h\mu_{g'}) \notin M_\sigma^{\textrm{rel}}$; define $h':=k^{-1}hk$ and observe that $h' \in H$ since $H \unlhd G$. If we look at the monodromy action of $h'$ on the abelian logarithm we obtain
    \begin{align*}
        \log_\sigma &\xmapsto{k} \log_\sigma + v_1\omega_1 + \cdots + v_{2g} \omega_{2g} \xmapsto{h} \log_\sigma + v_1\omega_1 + \cdots + v_{2g} \omega_{2g} + u_1^h\mu_1 + \cdots + u_{g'}^h\mu_{g'}\\
        & \xmapsto{k^{-1}} \log_\sigma +\rho(k^{-1})\cdot(u_1^h\mu_1 + \cdots + u_{g'}^h\mu_{g'}).
    \end{align*}
    Since $h' \in H$, then we have $\rho(k^{-1})\cdot (u_1^h\mu_1 + \cdots + u_{g'}^h\mu_{g'}) \in M_\sigma^{\textrm{rel}}$, which contradicts the choice of $k$ and $h$. This concludes the proof.
\qed

\subsection{The case of thin Monodromy of periods}\label{sec:3.5}
In \cite{And_mon}, André employs Hodge-theoretic methods to study the problems considered in this paper. Under the following hypotheses, he establishes  $M_\sigma^{\textrm{rel}} \cong \mathbb{Z}^{2g}$:
\begin{itemize}
    \item[$(i)$] $\cA \to S$ is without fixed part, even after pull-back to a finite \'etale cover of $S$.

    \item[$(ii)$] $\mathbb Z\sigma(S)$ is Zariski dense in $\mathcal A$.

    \item[$(iii)$] The monodromy group $\textrm{Mon}(\cA)$ is not thin.
\end{itemize}
Essentially, André's hypothesis $(iii)$ serves as an alternative to our assumption on the modular map~$p$. The geometric distinction is clear: while we require a certain ``sparseness'' condition on isomorphic fibers of the family, hypothesis $(iii)$ intuitively imposes that the image of the modular map is ``sufficiently large.'' A more precise comparison of these conditions, or a characterization of when they coincide, would be an interesting direction for further investigation.

Notice that our hypotheses in \Cref{mainThm2} don't exclude the case when $\textrm{Mon}(\cA)$ is thin. Indeed, as a consequence of \Cref{mainThm2} and the study of fiber products, we provide below an explicit construction showing a case where  $M_\sigma^\textrm{rel}$ is of maximal rank even though $\textrm{Mon}(\cA)$ is thin.

\begin{expl}\label{exa:thinMon}
    We consider the same construction as in \Cref{expl:tor_ram}. Note that the two elliptic schemes $\cE_1\to B$ and $\cE_2 \to B$ are not isogenous, since they have a diﬀerent set of bad reduction. The fundamental group $G$ of $B(\mathbb C)$ is free on three generators, and the monodromy representation on the periods of $\phi: \cA \to B$ gives a homomorphism $\rho:G \to \Gamma_2 \times \Gamma_2$, where
    \[
    \Gamma_2:=\left\langle
    \begin{pmatrix}
        1 & 2\\
        0 & 1
    \end{pmatrix},
    \;
    \begin{pmatrix}
        1 & 0\\
        2 & 1
    \end{pmatrix}
    \right\rangle \le \textrm{SL}_2(\mathbb Z).
    \]
    We denote the monodromy group of periods as $\textrm{Mon}(\cA):=\rho(G)$. Since $\phi:\cA \to B$ is the product of two non-isogenous elliptic schemes, then the Zariski closure of $\textrm{Mon}(\cA)$ is the whole $\textrm{SL}_2\times \textrm{SL}_2$ \cite[Theorem 5.3.10]{PONCELET}.

    Notice that $\textrm{Mon}(\cA)$ is generated by three elements, and the same holds for its projection on the abelianized group of $\Gamma_2 \times \Gamma_2$, which is $\mathbb Z^4$. In particular, the projection in the abelianized group must have infinite index. Therefore, the group $\textrm{Mon}(\cA)$ is not of finite index in the integral points of its Zariski-closure: in other words, the group $\textrm{Mon}(\cA)$ is thin.

    The abelian scheme $\phi\colon\cA\to B$ has trivial Mordell-Weil rank. Let us consider the map $\sigma:B \to \cA$ defined by
    \[
    \sigma(\lambda):=\left(\left(2,\sqrt{2(2-\lambda)}\right), \left(3, \sqrt{6(1+\lambda)}\right)\right).
    \]
    This map is a multisection of the abelian scheme $\phi$, but it becomes a section whether we consider the base change induced by the degree $4$ field extension $\bC(B) \subseteq \bC(B)(\psi,\mu)$ where $\psi^2 + \mu^2 = 3$ and $\mu^2=1+\lambda$. Denote by $\phi': \cA' \to B'$ the base change we just described. The monodromy of periods $\textrm{Mon}(\cA')$ is a finite index subgroup of $\textrm{Mon}(\cA)$: in particular, it is Zariski dense in $\textrm{SL}_2\times \textrm{SL}_2$ and it has infinite index in $\Gamma_2 \times \Gamma_2$; so $\textrm{Mon}(\cA')$ is still thin. The morphism $\sigma$ is a section of $\phi'$ whose image is not contained in any proper group subscheme, thus the relative monodromy group $M_\sigma^\textrm{rel}$ is isomorphic to $\bZ^4$ by \cite[Theorem 4.11]{T1}. Notice that $\sigma$ is ramified, since the field extension defining $B'$ induces a ramified morphism onto $B$ and therefore this gives rise to a ramified (finite flat) modular map.

\end{expl}

\section{Some applications}\label{section4}
We present some useful applications of our results. As a consequence of the non-triviality of relative monodromy we obtain a strong version of Manin's kernel theorem under some hypothesis on unramified sections. In the case of elliptic schemes over a curve the same result was obtained in \cite{CZ}. We refer to the Betti coordinates defined in \Cref{perAbLog} (see \Cref{log}).

\begin{thm}
    Let $\phi:\cA \to S$ be an abelian scheme without fixed part such that the modular map $p:S \to T \subseteq \mathbb A_g$ is generically finite. Denote by $U_p \subseteq S$ the Zariski open subset on which $p$ is finite and flat. Assume that all the sections unramified over $U_p$ are torsion. Let $\sigma:S \to \cA$ be a section. If the Betti coordinates $\beta_{\sigma,i}$ with $i=1, \ldots, 2g$ are globally defined functions on $S(\bC)$, then $\sigma$ is a torsion section. In that case, the Betti coordinates are rational constants.
\end{thm}

\begin{proof}
    If the Betti coordinates are globally defined on $S(\bC)$ then the relative monodromy group $M_\sigma^{\textrm{rel}}$ is trivial. The claim then follows by \Cref{mainThm1}.\\
\end{proof}

Our theorem about the rank of the relative monodromy group \Cref{mainThm2} gives a different proof of a result due to André \cite[Theorem 3]{André} about the algebraic independence of logarithms under our hypothesis on the abelian scheme $\cA \to S$. Let's briefly describe the setting.

Recall that the monodromy action of the fundamental group $\pi_1(S(\bC),s)$ on periods and the logarithm of a section $\sigma$ induces a representation $\theta_\sigma: \pi_1(S(\bC),s) \rightarrow \textrm{SL}_{2g+1}(\bZ)$ described in \Cref{simultaneousRepresentation}. Consider the projection $\theta_\sigma(\pi_1(S(\bC),s)) \to \textrm{Sp}_{2g}(\bZ)$, whose image is the monodromy group $\textrm{Mon}(\cA)$. By general theory, the Zariski closure $\overline{\theta_\sigma(\pi_1(S(\bC),s))}$ in $\textrm{SL}_{2g+1}(\mathbb Z)$ is the differential Galois group of the Picard-Vessiot extension of $\bC(S)$ generated by the coordinates of $\omega_1, \ldots, \omega_{2g}, \log_\sigma$ and their directional derivatives along a tangent vector field $\partial$ with respect to the Gauss-Manin connection. Clearly the homomorphism $\theta_\sigma(\pi_1(S(\bC),s)) \to \textrm{Sp}_{2g}(\bZ)$ extends to an algebraic group homomorphism
\[
\xi: \overline{\theta_\sigma(\pi_1(S(\bC),s))} \to \textrm{Sp}_{2g}(\mathbb Z)
\]

\begin{thm}
    Let $\phi:\cA \to S$ be an abelian scheme without fixed part such that the modular map $p:S \to T \subseteq \mathbb A_g$ is generically finite. Denote by $U_p \subseteq S$ the Zariski open subset on which $p$ is finite and flat. Assume that the action of the monodromy group $\textrm{Mon}(\cA)$ is irreducible. If a section $\sigma:S \to \cA$ is ramified over $U_p$, then the kernel $\ker{\xi}$ is isomorphic to $\mathbb{G}_a^{2g}$. In particular, the coordinates of $\log_\sigma$ have transcendence degree $g$ over $\bC(S)(\omega_1, \ldots, \omega_{2g})$, where $\bC(S)(\omega_1, \ldots, \omega_{2g})$ denotes the extension of $\bC(S)$ generated by the coordinates of periods.
\end{thm}

\begin{proof}
    By \Cref{mainThm2} the kernel of the morphism $\xi|_{\theta_\sigma(\pi_1(S(\bC),s))}$ is isomorphic to $\mathbb{G}_a^{2g}(\bZ)$, which implies the first part of the statement. In particular, the transcendence degree of the extension generated by the coordinates of $\log_\sigma$ and their directional derivatives over $\bC(S)(\omega_1, \ldots, \omega_{2g}, \partial\omega_1, \ldots, \partial \omega_{2g})$ is $2g$, where we are fixing a tangent vector field $\partial$. The conclusion follows.\\
\end{proof}

\bibliographystyle{plainurl}

\bibliography{biblio}
 
\Addresses

\end{document}